\documentclass[12pt]{amsart}
\usepackage{tikz}
\usepackage{amsmath, mathabx}
\usepackage{amsxtra}
\usepackage{amscd}
\usepackage{amsthm}
\usepackage{amsfonts}
\usepackage{amssymb}
\usepackage {pstricks}
\usepackage {tikz}
\usepackage{pstricks,pst-node}
\usepackage[all]{xy}
\usepackage{mathbbol}

\usepackage[plainpages,backref,urlcolor=blue]{hyperref}

\newtheorem{theorem}{Theorem}[section]

\newtheorem{lemma}[theorem]{Lemma}

\newtheorem{corollary}[theorem]{Corollary}
\newtheorem{proposition}[theorem]{Proposition}

\theoremstyle{definition}
\newtheorem{definition}[theorem]{Definition}

\theoremstyle{remark}
\newtheorem{remark}[theorem]{Remark}

\numberwithin{equation}{section}

\newcommand{\mc}{\mathcal}

\newcommand{\C}{{\mathbb C}}

\newcommand{\Z}{{\mathbb Z}}

\newcommand{\CB}{{\mathcal B}}
\newcommand{\CE}{{\mathcal E}}
\newcommand{\CF}{{\mathcal F}}
\newcommand{\CG}{{\mathcal G}}

\newcommand{\CO}{{\mathcal O}}
\newcommand{\CK}{{\mathcal K}}
\newcommand{\CM}{{\mathcal M}}
\newcommand{\CT}{{\mathcal T}}

\newcommand{\id}{{\rm{id}}}

\newcommand{\mf}{\mathfrak}
\newcommand{\fg}{{\mf g}}

\newcommand{\fb}{{\mf b}}

\newcommand{\cB}{\mc B}
\newcommand{\cC}{\mc C}

\newcommand{\be}{\begin{equation}}
\newcommand{\ee}{\end{equation}}

\newcommand{\cM}{\mc M}

\newcommand{\inv}{^{-1}}

\newcommand{\half}{\frac{1}{2}}
\newcommand{\is}{\sqrt{-1}}

\newcommand{\lr}{\longrightarrow}

\newcommand{\U}{{\rm{U}}}
\newcommand{\End}{{\rm{End}}}
\newcommand{\Hom}{{\rm{Hom}}}

\newcommand{\Uq}{{{\rm U}_q}}
\newcommand{\GL}{{\rm{GL}}}
\newcommand{\TL}{{\rm{TL}}}
\newcommand{\Sym}{{\rm{Sym}}}

\newcommand{\rank}{{\rm{rank}}}

\newcommand{\tl}{Temperley-Lieb }

\advance\headheight by 2pt

\newcommand{\fsl}{{\mathfrak {sl}}}

\newcommand{\fgl}{{\mathfrak {gl}}}
\newcommand{\gl}{{\mathfrak {gl}}}

\newcommand{\ot}{\otimes}

\newcommand{\ve}{{\varepsilon}}

\newcommand{\HH}{\widehat{H}}
\newcommand{\HTL}{\widehat{\rm TL}}
\newcommand{\HBMW}{\widehat{\rm BMW}}

\newcommand{\ATLC}{{\widehat{\mathcal{TL}}}}
\newcommand{\TLC}{\mathcal{TL}}
\newcommand{\TLB}{{\rm{TLB}}}
\newcommand{\TLBB}{{\mathbb{TLB}}}
\newcommand{\TLBC}{\mathcal{TLB}}
\newcommand{\RTC}{\mathcal{RT}}
\newcommand{\qint}[1]{\ldbrack{#1}\rdbrack_q }

\begin{document}

\normalfont

\title[Schur-Weyl duality]{
Schur-Weyl duality for certain \\ infinite dimensional $\U_q(\fsl_2)$-modules}
\author{K. Iohara, G.I. Lehrer and R.B. Zhang} 
\thanks{Partially supported by the Australian Research Council.}
\thanks{The authors thank the Research Institute at Oberwolfach
for its hospitality at an RIP program, where this work was begun.}
\subjclass[2010]{17B37, 20G42 (primary) 81R50 (secondary)}
\keywords{Tangle category, Temperley-Lieb category of type $B$, Verma module.}
\address{Univ Lyon, Universit\'{e} Claude Bernard Lyon 1, CNRS UMR 5208, Institut Camille Jordan, 
43 Boulevard du 11 Novembre 1918, F-69622 Villeurbanne cedex, France}
\email{iohara@math.univ-lyon1.fr}
\address{School of Mathematics and Statistics,
University of Sydney, N.S.W. 2006, Australia}
\email{gustav.lehrer@sydney.edu.au, ruibin.zhang@sydney.edu.au}


\begin{abstract} 
Let $V$ be the two-dimensional simple module and $M$ be a projective Verma module for the quantum group of $\mathfrak{sl}_2$ at generic $q$.  
We show that for any $r\ge 1$, the endomorphism algebra of $M\otimes V^{\otimes r}$ is isomorphic to the type $B$ Temperley-Lieb 
 algebra $\TLB_r(q, Q)$ for an appropriate parameter $Q$ depending on $M$. The parameter $Q$ is determined explicitly. We also use the cellular 
 structure to determine precisely for which values of $r$ the endomorphism algebra is semisimple. A key element of our method is to identify 
 the algebras $\TLB_r(q,Q)$ as the endomorphism algebras of the objects 
 in a quotient category of the category of coloured ribbon graphs of Turaev and Reshitikhin or the tangle diagrams of Freyd-Yetter.
\end{abstract}

\maketitle

\tableofcontents

\section{Introduction} \label{sect:intro}

Let $V$ be the two-dimensional simple module and $M$ be a projective Verma module for the quantum group $\U_q(\fsl_2)$ at generic $q$.  
We show that the endomorphism algebra of $M\otimes V^{\otimes r}$ is isomorphic to the type $B$ \tl algebra $\TLB_r(q, Q)$ of degree $r$ 
with an appropriate parameter $Q$ depending on $M$. 
We derive this from a more general result, which states that the full subcategory $\CT$ of the category of $\U_q(\fsl_2)$-modules consisting 
of the objects $M\otimes V^{\otimes r}$ for all $r\in\Z_{\ge 0}$ is isomorphic to the \tl category of type $B$ \cite{GL03} with parameters $q,Q$.  
The \tl algebra of type $B$ is the endomorphism algebra of an object of this category, 
in the same way as the Brauer algebra in the Brauer category \cite{LZ14}. 

There are two variants of the \tl category which we meet; one is defined as a subquotient of the coloured framed
 tangle category extensively studied by Freyd, Yetter \cite{FY}, Turaev, Reshitikhin \cite{RT, RT-1, T} 
 and others. This is used to prove the result referred to above.
Another version of the  \tl category of type $B$ which we denote by $\TLBB(q,Q)$ was first introduced in \cite{GL03};  
we briefly explain it in Section \ref{sect:TLB-old},
where we show that the two versions are equivalent, with appropriate parameter matchings. 

This enables us to bring into play the cellular structure of the \tl algebra, which we use to determine precise criterion for 
the semisimplicity of the endomorphism algebra.

To explain the main idea of this new formulation, we recall that a uniform description of the affine Hecke algebra, affine \tl 
algebra and other related algebras was given in \cite{GL03}, where the unifying object is the braid group of type $B$. 
The affine \tl algebra is a quotient of the group algebra of this group, which factors through the affine Hecke algebra,
 while the  \tl  algebra of type $B$ is a further quotient of the affine \tl algebra. Now by embedding the  braid group of 
 type $B$ of $r$-strings into the braid group of type $A$ of $(r+1)$-strings using Lemma \ref{lem:aff-fin}, we arrive at 
 the diagrammatic description of the affine \tl algebra etc. in terms of polar tangle diagrams \cite{A, DRV}.  

One of the advantages of such diagrammatics is that it allows one to construct representations of the braid group of type $B$ 
and quotients of its group algebra using the theory of quasi triangular Hopf algebras such as quantum groups 
 (see, e.g., \cite{DRV, LZ06} and also Theorem \ref{thm:aff-Hecke-reps}).

Another advantage, which is more important to us here, is that it allows us to adapt techniques from quantum topology \cite{T} to 
develop categorical formulations of the group algebra of the braid group of type $B$ and its quotient algebras, and relate
 the resulting categories to categories of representations of quantum groups. We follow 
this strategy to develop a new version of the \tl category of type $B$. As we will see, several new categories also arise 
naturally from this process,  which may be interesting in their own right. 

We begin in Section \ref{sect:cat-RT} with a category $\RTC$ of coloured un-oriented tangle diagrams up to regular isotopy
 \cite{FY, RT, T}. The objects of the category are sequences of elements of $\cC:=\{m, v\}$, and the modules of morphisms are spanned
  by a class of coloured un-oriented tangle diagrams up to regular isotopy.  We study several quotient categories of $\RTC$, including
  the \tl category of type $B$  and some related new categories. Their endomorphism algebras are interesting in their own right, and are closely
related to endomorphism algebras of certain categories of representations \cite{ALZ, Ro, Ze}. 

The affine \tl category $\ATLC(q)$ is the quotient of $\RTC$ obtained by imposing on morphisms the skein relations \eqref{eq:skein1} 
and \eqref{eq:skein2}, and the free loop removal relation \eqref{eq:flr}.
The module $\HTL^{ext}_r(q):=\Hom_{\ATLC(q)}((m, v^r), (m, v^r)))$ of endomorphisms of the object 
$(m, v^r):=(m, \underbrace{v, \dots, v}_r)$ in $\ATLC(q)$  forms an associative algebra (see Definition \ref{def:ext-ATL}), 
which is generated by a subalgebra isomorphic to the affine \tl algebra $\HTL_r(q)$ together with two additional generators which are central. 
 This algebra has a more tractable structure than $\HTL_r(q)$ as we will see below. 

The one parameter multi-polar \tl category $\ATLC(q, \Omega)$ is obtained as a quotient of $\ATLC(q)$ by specialising 
the central elements mentioned above to appropriate scalars related to $\Omega$ (Definition \ref{def:ATLC-q}). 
 The  \tl category $\TLBC(q, \Omega)$ of type $B$ is a full subcategory of $\ATLC(q, \Omega)$ with objects of the 
 form $(m, v^r)$ for all $r\in\Z_{\ge 0}$. The category $\ATLC(q, \Omega)$ contains the finite \tl category $\TLC(q)$ as a 
 full subcategory in two different ways (Remark \ref{rem:tlsub}); one of these
 is also contained in $\TLBC(q, \Omega)$. 

The interrelationships among the categories mentioned above and other categories which arise naturally in reformulating the  \tl 
category of type $B$ are recapitulated in Section \ref{sect:summary} and illustrated in Figure \ref{fig:cats}.

The structure of the category $\TLBC(q, \Omega)$ (in its two-parameter version) is studied in detail in 
Section \ref{sect:TLBC-struct}. In particular, we are able to determine explicitly the dimensions of the morphism 
spaces at  generic $q$.  The category $\TLBC(q, \Omega)$ is quite different from the 
\tl category $\TLBB(q,Q)$ of type $B$ introduced in \cite{GL03}. Nevertheless, we give a 
direct proof that $\TLBC(q, \frac{\Omega}{\sqrt{-1}})$ and $\TLBB(q,Q)$ are isomorphic in Section \ref{sect:TLB-old}. 

We construct in Theorem \ref{thm:RT} a tensor functor $\widehat\CF$ from the category $\RTC$ to the category 
$\CO_{int}$ of finitely generated integral weight $\U_q(\fsl_2)$-modules of type-${\bf 1}$, which are locally
$\U_q(\fb)$-finite   (where $\U_q(\fb)$ is 
the quantum Borel subalgebra). This functor factors through the category $\ATLC(q, \Omega)$ of type $B$ 
with the parameter $\Omega$, whose  dependency on $M$ is given by \eqref{eq:omega-value}. This induces a functor from $\ATLC(q, \Omega)$ to $\CO_{int}$,
 which restricts to a functor $\CF': \TLBC(q, \Omega) \longrightarrow \CT$, where $\CT$ is the full 
 subcategory of $\CO_{int}$ mentioned above. 
Since $M$ is a projective Verma module for $\U_q(\fsl_2)$, the  structure of the category $\CT$ is relatively easy to understand. 
 Putting together the structural information for $\TLBC(q, \Omega)$  and $\CT$,  we are able to show  in 
 Theorem \ref{thm:main} that $\CF'$ is an isomorphism of categories.

The categorical development here also leads to improved understanding of the theory of the affine \tl algebra. 
For example, Lemma \ref{lem:central-skein} shows that  the translation generator of 
the  $\HTL_r(q)$ subalgebra of  $\HTL^{ext}_r(q)$ 
satisfies a $3$-term skein relation over the centre. The skein relations involve the additional (central) generators, 
thus is not a relation inside the $\HTL_r(q)$ subalgebra.  This gives a conceptual explanation of the fact \cite{GL03} 
that the translation generator always has a characteristic polynomial of degree $2$  in the cell  and irreducible 
representations of the affine \tl algebra (see Remark \ref{rem:skein-universal}).

Throughout the paper, we work over the ground field $\CK_0:=\C(q^{\frac{1}{k}})$, where $q$ is an indeterminate over $\C$ such that $(q^{\frac{1}{k}})^k=q$
for some fixed positive integer $k$.

\section{The affine Temperley-Lieb algebra}\label{sect:ATL-algebra}

We begin with a brief discussion of the affine Temperley-Lieb algebra and related algebras following \cite{GL98, GL03}.  
The unifying object is the braid group of type $B$. 
 Among the quotients of its group algebra are the extended affine Hecke algebra of type $A$, the extended affine Temperley-Lieb algebra, the affine BMW algebra
 and the Temperley-Lieb algebras $\TLB_r(q,Q)$ of type $B$.

\subsection{The Artin braid group $\Gamma_r$ of type $B$}

Fix an integer $r\ge 2$. 
The braid group  of type $A$  on $r$ strings is denoted by $B_r$ and  has the following standard presentation. 
It has generators $b_1, \dots, b_{r-1}$, and defining relations 
\begin{eqnarray}\label{eq:braid}
\begin{aligned}
b_i b_{i+1} b_i = b_{i+1} b_i b_{i+1}, \quad b_i b_j = b_j b_i, \quad |i-j|>1.
\end{aligned}
\end{eqnarray}
The braid group $\Gamma_r$ of type $B$ \cite{GL03} is generated by $\{\sigma_i,  \xi_1\mid 1\le i\le r-1\}$ with the following relations:  
\begin{enumerate}
\item 
the $\sigma_i$ satisfy 
the standard braid relations \eqref{eq:braid};   
\item $\xi_1$ commutes with all $\sigma_j$ for $j\ge 2$, and 
\begin{eqnarray}\label{eq:sigma-xi}
\sigma_1 \xi_1 \sigma_1 \xi_1 = \xi_1 \sigma_1 \xi_1\sigma_1. 
\end{eqnarray}
\end{enumerate}
Define $\xi_{i+1} := \sigma_i \xi_i \sigma_i$ for $1\le i\le r-1$. 
Then we have  (see \cite[Prop. (2.6)]{GL03})
\[
\begin{aligned}
\xi_i \xi_j=\xi_j\xi_i, \quad \sigma_j \xi_i = \xi_i \sigma_j, \quad j>i. 
\end{aligned}
\]
It is known that $\Gamma_r$ contains $B_r$ as a subgroup, and it also clearly contains the subgroup $\Z^r$ generated by the elements $\xi_i$. 

The following result appears to be well-known.
\begin{lemma} \label{lem:aff-fin}
Let $B_{r+1}$ be the type $A$ braid group on $r+1$ strings, generated by $b_0, b_1, \dots, b_{r-1}$ subject to the standard relations (see \eqref{eq:braid}).  
There exists an injective group homomorphism given by 
\[
\eta_r: \Gamma_r \longrightarrow B_{r+1}, \quad \xi_1\mapsto b_0^2, \quad \sigma_i \mapsto b_i, \quad 1\le i\le r-1.  
\]
\end{lemma}
It is a pleasant exercise to verify directly that $\eta_r$ preserves the relation \eqref{eq:sigma-xi}.  
We have 
\[
\begin{aligned}
\eta_r(\sigma_1 \xi_1 \sigma_1 \xi_1 )&=b_1 b_0^2 b_1 b_0^2 = b_1 b_0 (b_0 b_1 b_0) b_0  \\
&=   (b_1 b_0 b_1) (b_0 b_1 b_0)=b_0 (b_1 b_0 b_1) b_0 b_1\\
& = b_0 ^2 b_1 b_0 ^2 b_1= \eta_r(\xi_1 \sigma_1 \xi_1\sigma_1).
\end{aligned}
\]

\begin{definition}\label{def:grprings}
We write $\cB_r$ for the group ring $\CK_0 B_r$ and $\CG_r$ for the group ring $\CK_0 \Gamma_r$. Of
course $\cB_r\subset \CG_r$.
\end{definition}

\subsection{Quotient algebras of the group algebra of $\Gamma_r$}
\subsubsection{The (extended) affine Hecke algebra of type $A$}

Let $J_r$ be the two-sided ideal of the group algebra  $\CG_r$ of the braid group of type B generated by $(\sigma_i -q)(\sigma_i +q^{-1})$ for all $i$.
The affine Hecke algebra $\HH_r(q)$ (cf. \cite[Def. (3.1)]{GL03}) is the quotient algebra of $\CG_r$ by the ideal $J_r$:
$
\HH_r(q):= \CG_r/J_r.
$
Denote by $T_i$ and $X_j$ respectively the images of $\sigma_i$  and $\xi_j$  in $\HH_r(q)$.  Then the elements $T_i$ ($1\le 1\le r-1$) and $X_j$ ($1\le j\le r$)
satisfy the following relations.
\be\label{eq:ha}
\begin{aligned}
&T_i T_j = T_j T_i, \quad |i-j|>1, \\
&T_i T_{i+1} T_i = T_{i+1} T_i T_{i+1}, \\
&(T_i -q)(T_i +q^{-1})=0, \\
&X_{i+1} = T_i X_i T_i, \quad X_i X_j =X_j X_i.
\end{aligned}
\ee

Moreover the relations \eqref{eq:ha} form a set of defining relations of $\HH_r(q)$.

The elements $T_i$ generate a subalgebra of $\HH_r(q)$, which is isomorphic to the finite dimensional Hecke algebra $H_r(q)$
 of type  $A_{r-1}$, with $ \CK_0$-basis $\{T_w\mid w\in\Sym_r\}$. Note that  
$H_r(q)$ is the quotient of $\CB_r$ by the two-sided ideal  $J_r\cap\CB_r$.

One deduces easily from the relations \eqref{eq:ha} that for any Laurent polynomial $f=f(X^{\pm 1}_1, \dots, X^{\pm 1}_r)$,
\begin{eqnarray}\label{eq:Bernstein-relation}
T_i f - f^{s_i} T_i = (q-q^{-1}) \frac{f - f^{s_i}}{1-X_i X_{i+1}^{-1}},
\end{eqnarray}
where  $f^{s_i}$ is obtained from $f$ by interchanging $X^{\pm 1}_i$ and $X^{\pm 1}_{i+1}$.  
From this relation it is evident that the symmetric Laurent polynomials in the elements $X_i^{\pm 1}$ are all 
central elements of $\HH_r(q)$, and it is known \cite{L} that they generate the center.

The elements $X_i^{\pm 1}$ generate the Laurent polynomial ring $\CK_0[X^{\pm 1}_1, \dots, X^{\pm 1}_r]$.  We have the vector space isomorphism
$
\HH_r(q)=H_r(q)\otimes_{\CK_0}\CK_0[X^{\pm 1}_1, \dots, X^{\pm 1}_r].
$

\subsubsection{The affine Temperley-Lieb algebra}\label{ss:atl}
Let $\langle \CE_3\rangle$ be the two-sided ideal of $\HH_r(q)$ generated by 
$
\CE_3=\sum\limits_{w\in\Sym_3}(-q^{-1})^{\ell(w)} T_w, 
$
where $\Sym_3=\Sym\{1, 2,3\}$ is the symmetric group on $\{1, 2,3\}$, 
regarded as a subgroup of $\Sym_r$ on $\{1, 2, \dots, r\}$.  For $i\in\{ 3, \dots,  r\}$, let 
$
\CE_i=\sum\limits_{w}(-q^{-1})^{\ell(w)} T_w, 
$
where the sum is over the elements of $\Sym\{i, i-1, i-2\}$.  Then $\CE_i\in\langle \CE_3\rangle$ for all $i$.

We define the affine Temperley-Lieb algebra \cite{GL03} by
\be\label{eq:defhtl}
\HTL_r(q) = \HH_r(q)/ \langle \CE_3\rangle. 
\ee

Let $C_i=T_i-q\in\HH_r(q)$ ($i=1,\dots, r-1$). These are the Kazhdan-Lusztig basis elements corresponding
to the simple reflections in $\Sym_r$, and satisfy $C_i^2=-(q+q\inv)C_i$. 
Moreover, applying  the automorphism $T_w\mapsto(-1)^{\ell(w)}(T_{w\inv})\inv$ of $\HH_r(q)$, 
to the relation at the bottom of
\cite[p. 487]{GL03}, we see, using the fact that this automorphism maps $E_i$ of {\it loc. cit.} to $q^6\CE_{i}$, that
\be\label{eq:tlc}
C_iC_{i+1}C_i-C_i=C_{i+1}C_iC_{i+1}-C_{i+1}=-q^3\CE_3.
\ee

Denote the image of $C_i$ in $\HTL_r(q)$ by $ e_i$ (so that $T_i\mapsto e_i+q$) and the image of $X_i$ by $x_i$.  
Then $\HTL_r(q)$ is generated by $e_i$ ($1\le i \le r-1$) and $x_i^{\pm 1}$ ($1\le i \le r$) 
subject to the following relations.
\be\label{eq:atlrel}
\begin{aligned}
&e_i e_j = e_j e_i, \quad |i-j|>1, \\
&e_i^2 = - (q+q^{-1}) e_i, \quad e_i e_{i\pm 1} e_i = e_i, \\
&x_{i+1} = (q+e_i) x_i (q+e_i), \quad x_i x_j =x_j x_i.\\
\end{aligned}
\ee
It follows from the above relations  that for any Laurent polynomial $f$ in $x^{\pm 1}_1, \dots, x^{\pm 1}_r$, 
\[
e_i f - f^{s_i} e_i = (qx_i x^{-1}_{i+1}-q^{-1}) \frac{f - f^{s_i}}{1-x_i x_{i+1}^{-1}}.
\]

The relation $x_1x_2=x_2x_1$ is also easily shown to amount to the relation 
\begin{eqnarray}\label{eq:constraints}
q e_1x_1^2+e_1x_1e_1x_1 =qx_1^2e_1+x_1e_1x_1e_1.  
\end{eqnarray}
This implies the following result. 
\begin{lemma}\label{lem:extra-inv}
Let  $\delta=-(q+q^{-1})$.  Then the following holds in $\HTL_r(q)$.
\begin{eqnarray}\label{eq:extra-inv}
\delta(q e_1x_1^2+e_1x_1e_1x_1) =q e_1 x_1^2e_1+ e_1 x_1e_1x_1e_1
= \delta(qx_1^2e_1+x_1e_1x_1e_1). 
\end{eqnarray}
\end{lemma}
\begin{proof}
Multiplying \eqref{eq:constraints} by $e_1$  on the left (resp. right), we obtain the first (resp. second) equality. 
\end{proof}

\begin{remark} The significance of the lemma is best appreciated when placed in the context of the extended 
affine Temperley-Lieb algebra $\HTL^{ext}_r(q)$ defined in Definition \ref{def:ext-ATL}, 
where \eqref{eq:extra-inv} may be understood as a ``skein relation'' imposed on the centre of $\HTL^{ext}_r(q)$, 
(see Lemma \ref{lem:central-skein}). 
\end{remark}

Another presentation of $\HTL_r(q)$ of which we shall make use below is as follows. 
Recall that the element $\tau:=\xi_1b_1b_2\dots b_{r-1}\in\Gamma_r$
is represented as the ``twisting braid'', and satisfies $\tau b_i\tau\inv=b_{i+1}$, where the subscript $i$ is taken mod $r$.
Denote by $V$ the image of $\tau$ in $\HH_r(q)$ and in $\HTL_r(q)$. Then $\HTL_r(q)$ is generated by 
$e_1,\dots,e_r,V$ with relations
\be\label{eq:atlrel2}
\begin{aligned}
&e_i e_j = e_j e_i, \quad |i-j|>1, \\
&e_i^2 = - (q+q^{-1}) e_i, \quad e_i e_{i\pm 1} e_i = e_i, \\
&Ve_iV\inv=e_{i+1},\\
\end{aligned}
\ee
for $i=1,\dots,r$, and in the the index is taken mod $r$. 

Note that $(e_i+q)\inv=e_i+q\inv$ is the image of $T_i\inv$ in $\HTL_r(q)$.
\subsubsection{The Temperley-Lieb algebra of type $B$}
%
%


\begin{definition} \label{def:TLB-alg} \cite{GL03} Given an invertible scalar $Q$, let  $J_Q$ be the two-sided 
ideal of the affine Temperley-Lieb algebra $\HTL_r(q)$ generated by the elements
\[
 (x_1-Q)(x_1+Q^{-1}) \quad\text{and  }\; e_1 x_1 e_1 + q(Q-Q^{-1}) e_1.
\]
The Temperley-Lieb algebra of type $B$, denoted by $\TLB_r(q, Q)$, is defined by 
\[
\TLB_r(q, Q):=\HTL_r(q)/J_Q.
\]
\end{definition}
We denote by $\pi_r(Q): \HTL_r(q) \longrightarrow \TLB_r(q, Q)$ the canonical surjection. 
We will also denote the images of the generators $x_1^{\pm 1}$ and $e_i$ in
$\TLB_r(q, Q)$ by the same symbols.  

\begin{remark}
The relations in \eqref{eq:extra-inv} are identically satisfied in $\TLB_r(q, Q)$. 
\end{remark}

Note that the subalgebra generated by $e_i$ ($i=1, ,2, \dots, r-1$) in $\TLB_r(q, Q)$ is isomorphic to the
usual  Temperley-Lieb algebra $\TL_r(q)$ on $r$ strings. 
\subsubsection{The affine BMW algebra}
The affine BMW algebra \cite{DRV} is another interesting quotient algebra of the group algebra $\CG_r$ of the braid group of type B. 
Fix non-zero scalars $z, y\in\CK_0$. Define elements $\theta_i\in\CG_r$ by 
\[
\theta_i:= 1-\frac{\sigma_i- \sigma_i^{-1}}{z}, \quad  i=1,2,\dots,r-1
\]
and let  $J_{BMW}$ be the two-sided ideal  generated by the elements 
\[
\begin{aligned}
&\theta_i \sigma_i - y \theta_i, \quad  \sigma_i \theta_i - y \theta_i, \quad
&\theta_i \sigma_{i-1}^{\pm 1} \theta_i - y^{\mp 1} \theta_i, \quad
\theta_i \sigma_{i+1}^{\pm 1} \theta_i - y^{\mp 1} \theta_i
\end{aligned}
\]
for all valid indices.  The affine BMW algebra $\HBMW_r(z, y)$ is defined by 
\[
\HBMW_r(z, y)= \CG_r/J_{BMW}.
\]
We will denote the images of 
$\sigma_i$ and $\theta_i$ in $\HBMW_r(z, y)$ by $g_i$ and $e_i$ respectively, 
and that of $\xi_i$ by $X_i$. Then $\HBMW_r(z, y)$ is generated by $g_i$, $e_i$ ($1\le i\le r-1$) and $X_j$ ($1\le j\le r$) subject to the following relations.
\begin{eqnarray*}
&\text{all $g_i$ and $X_j$ are  invertible}, \\
&g_i g_j = g_j g_i,  \quad |i-j|>1, \\
&g_i g_{i+1} g_i = g_{i+1} g_i g_{i+1}, \\
&e_i= 1 - \frac{g_i - g_i^{-1}}{z}, \\
&e_i g_i = g_i e_i = y e_i, \\
&e_i g_{i-1}^{\pm 1} e_i = e_i g_{i+1}^{\pm 1} e_i =y^{\mp 1} e_i, \\
&X_{i+1}= g_i X_i g_i, \quad X_i X_j =X_j X_i.
\end{eqnarray*}
It is easy to derive the following relations from the defining relations.
\[
e_i e_{i\pm 1} e_i = e_i, \quad  e_i^2 = \left(1- \frac{y - y^{-1}}{z}\right)e_i.
\] 

We will not discuss this algebra further, since we shall not need it in this paper. 

%
%
%
%
%
%
%
\subsection{A diagrammatic presentation of the braid group of type $B$}\label{sect:diagrams}

A widely used diagrammatic presentation of the braid group of type $B$ is in terms of braids on cylinders \cite{GL03}. 
We shall here adopt an alternative presentation which use braids with a pole \cite{A}. 
Let us represent 
the braid group $B_{r+1}$ of type $A$ in terms of braids in the standard way. Then the second presentation
of $\Gamma_r$ is obtained from this by regarding it as a subgroup of $B_{r+1}$ via the injection $\eta_r$ defined by 
Lemma \ref{lem:aff-fin}, where the $0$-th string is considered to be a pole. We can also turn a cylindrical braid into a braid
 with a pole by regarding the core of the cylinder as the pole and pushing it to the far left.

To distingush elements of $\Gamma_r$ and $B_{r+1}$,  we draw a braid of type $B$ as a diagram 
consisting of a pole and $r$-strings on its right, where a string can only cross the pole  an even number of times. 
The strings will be drawn as thin arcs, and the pole drawn as a vertical thick arc.  
The generators of $\Gamma_r$ can be depicted as follows.  
\[
\begin{aligned}
&
\begin{picture}(150, 70)(-20,0)
\put(-45, 28){$\sigma_i=$}
{
\linethickness{1mm}
\put(-15, 0){\line(0, 1){60}}
}
\put(0, 0){\line(0, 1){60}}
\put(40, 0){\line(0, 1){60}}
\put(18, 30){...}
\qbezier(60, 0)(60, 0)(68, 27)
\qbezier(72, 33)(72, 33)(80, 60)
\qbezier(60, 60)(70, 30)(80, 0)
\put(100, 0){\line(0, 1){60}}
\put(120, 0){\line(0, 1){60}}
\put(105, 30){...}
\put(56, -10){\small$i$}
\put(72, -10){\small{$i$+1}}
\put(130, 0){, }
\end{picture}
\\
&
\begin{picture}(150, 70)(-20,0)
\put(-52, 28){$\sigma_i^{-1}=$}
{
\linethickness{1mm}
\put(-15, 0){\line(0, 1){60}}
}
\put(0, 0){\line(0, 1){60}}
\put(40, 0){\line(0, 1){60}}
\put(18, 30){...}

\qbezier(60, 60)(60, 60)(68, 33)
\qbezier(80, 0)(80, 0)(72, 27)

\qbezier(60, 0)(70, 30)(80, 60)
\put(100, 0){\line(0, 1){60}}
\put(120, 0){\line(0, 1){60}}
\put(105, 30){...}
\put(56, -10){\small$i$}
\put(72, -10){\small{$i$+1}}
\put(130, 0){, }
\end{picture} 
\end{aligned}\\
\]
\[
\begin{aligned}
&
\begin{picture}(150, 70)(-20,0)
\put(-48, 28){$\xi_1=$}
{
\linethickness{1mm}
\put(-15, 0){\line(0, 1){18}}
\put(-15, 22){\line(0, 1){38}}
}
\qbezier(0, 60)(-7, 50)(-12, 42)
\qbezier(-18, 38)(-25, 30)(-15, 20)
\qbezier(-15, 20)(-15, 20)(0, 0)
\put(20, 0){\line(0, 1){60}}
\put(40, 0){\line(0, 1){60}}
\put(60, 30){............}
\put(120, 0){\line(0, 1){60}}
\put(130, 0){, }
\end{picture}
\\
&
\begin{picture}(150, 70)(-20,0)
\put(-55, 28){$\xi_1^{-1}=$}
{
\linethickness{1mm}
\put(-15, 0){\line(0, 1){38}}
\put(-15, 42){\line(0, 1){18}}
}

\qbezier(-16, 40)(-25, 30)(-18, 20)
\qbezier(-16, 40)(-16, 40)(5, 60)
\qbezier(-12, 17)(-12, 17)(5, 0)

\put(20, 0){\line(0, 1){60}}
\put(40, 0){\line(0, 1){60}}
\put(60, 30){............}
\put(120, 0){\line(0, 1){60}}
\put(130, 0){.}
\end{picture}
\end{aligned}\\
\]

The multiplication of $\Gamma_r$ is then given by concatenation of diagrams. 
Given two diagrams $D$ and $D'$, their concatenation $D'\circ D$ is 
obtained by composing the diagrams 
$D$ and $D'$ by joining the points on the bottom of $D$ with the points on the top of $D'$.

It is easy to see that  
the elements $\xi_j$ ($1\le j \le r$) can be depicted as in Figure \ref{fig:xi}. 

\begin{figure}[h]
\begin{center}
\begin{picture}(150, 100)(-30,0)
\put(-48, 45){$\xi_j=$}
{
\linethickness{1mm}
\put(-15, 0){\line(0, 1){25}}
\put(-15, 30){\line(0, 1){60}}
}
\put(-5, 0){\line(0, 1){17}}
\put(-5, 22){\line(0, 1){68}}
\put(-5, 29){\line(0, 1){61}}
\put(15, 0){\line(0, 1){5}}
\put(15, 10){\line(0, 1){80}}
\put(0, 45){...}
\qbezier(-18, 60)(-40, 30)(30, 0)
\qbezier(17, 82)(25, 87)(30, 90)

\qbezier(13, 80)(0, 72)(-3, 70)
\qbezier(-7, 68)(-8, 67)(-12, 64)

\put(45, 0){\line(0, 1){90}}
\put(58, 45){......}
\put(90, 0){\line(0, 1){90}}
\put(27, -10){\small$j$ }
\put(100, 0){. }
\end{picture}
\end{center}
\caption{Diagram for $\xi_j$}
\label{fig:xi}
\end{figure}

From the diagrammatical presentation of the braid group of type $B$, we can derive diagrammatical
 presentations for the quotient algebras of the group algebra of $\CG_r$ discussed above.

%
\subsection{Quantum groups}

Let $A=(a_{i j})$ be the Cartan matrix of a finite dimensional simple 
Lie algebra or  affine
 Lie algebra $\fg$.  If $\{\alpha_i\vert\, i=1, 2, \dots\ell\}$ is the set of simple 
roots of $\fg$, the Cartan matrix 
$A$ is defined by $a_{i j} = \frac{2(\alpha_i, \alpha_j)}{(\alpha_i, \alpha_i)}$, 
where the bilinear form on 
the weight space if normalised so that $(\alpha_i, \alpha_i)=2$ for the short 
simple roots. Let $q_i=q^{(\alpha_i, \alpha_i)/2}$ for all $i$. 

The quantum group $\U_q(\fg)$ over $\CK_0$ is generated by 
$\{E_i, F_i, K_i, K_i^{-1}\mid 1\le i\le \ell\}$ subject to the relations 
\begin{eqnarray}
& K_i K_i^{-1}=1, \quad K_i K_j = K_j K_i,  \label{eq:KK}\\
&K_i E_j K_i^{-1}= q_i^{a_{i j}} E_j, \quad K_i F_j K_i^{-1}= q_i^{-a_{i j}} F_j, \label{eq:KEKF}\\
&E_i F_j - F_j E_i = \delta_{i j}  \frac{K_i - K_i^{-1}}{q_i-q_i^{-1}}, \label{eq:EF}\\
&\sum\limits_{r=0}^{1-a_{i j}} (-1)^r \begin{bmatrix}1-a_{i j}\\ r\end{bmatrix}_{q_i} E_i^{1-a_{i j}-r} 
E_j E_i^r=0, \quad i\ne j, \label{eq:Serre-E}\\
&\sum\limits_{r=0}^{1-a_{i j}} (-1)^r \begin{bmatrix}1-a_{i j}\\ r\end{bmatrix}_{q_i} F_i^{1-a_{i j}-r} 
F_j F_i^r=0, \quad i\ne j,  \label{eq:Serre-F}
\end{eqnarray}
where $\begin{bmatrix}n\\ r\end{bmatrix}_q=\frac{[n]_q!}{[r]_q![n-r]_q!}$ and $[j]_q=\frac{q^j- q^{-j}}{q-q^{-1}}$.  

As is well known,  $\U_q(\fg)$ is a Hopf algebra with co-multiplication defined by
\[
\begin{aligned}
&\Delta(E_i)=E_i\otimes K_i + 1\otimes E_i, \\
&\Delta(F_i)=F_i\otimes 1 + K_i^{-1}\otimes F_i,\\
&\Delta(K_i)=K_i\otimes K_i. 
\end{aligned}
\]

By slightly modifying the definition above, we also obtain the quantum groups 
$\U_q (\gl_k)$  \cite{J}\cite[\S 6.1]{LZ06} and  $\U_q (\mathfrak{o}_k)$ \cite[\S8.1.2]{LZ06}.

Henceforth $\U_q (\fg)$ will denote either the quantum group of a finite dimensional simple Lie algebra $\fg$, or 
$\U_q (\gl_k)$ and  $\U_q (\mathfrak{o}_k)$.

We will consider only left $\U_q(\fg)$-modules of type ${\bf 1}$. 
Given two $\U_q(\fg)$-modules $V$ and $W$ of which at least one is finite dimensional, 
we denote by 
$
{R_{V, W}} : V\otimes W\longrightarrow V\otimes W
$
the $R$-matrix of $\U_q(\fg)$ on $V\otimes W$. 
Let $\check{R}_{V, W}=\tau_{V, W} R_{V, W}$, where $\tau_{V, W}$ is the permutation defined by 
\[
\tau_{V, W}: V\otimes W  \longrightarrow  W \otimes V, \quad v\otimes w \mapsto w\otimes v.
\]
Then $\check{R}_{V, W}\in\Hom_{\U_q(\fg)}(V\otimes W, W\otimes V)$ is an isomorphism.

\noindent{\bf Notation}.  It follows from the above that if $V$ and $W$ are $\U_q(\fg)$-modules, of which at
least one is finite dimensional, then $\check R_{W,V}\circ \check R_{V,W}\in\End_{\U_q(\fg)}(V\ot W)$.
We shall write $\check R_{V,W}^2$ for the endomorphism $\check R_{W,V}\circ \check R_{V,W}$.

\begin{proposition}\label{prop:braid}
Let $U, V, W$ be $\U_q(\fg)$-modules, at least two of which are finite dimensional. 
\begin{enumerate}
\item
The following relation holds in $\Hom_{\U_q(\fg)}(U\otimes V\otimes W,  W\otimes V\otimes U)$:
\begin{eqnarray}\label{eq:YBE}
\begin{aligned}
(\check{R}_{V, W}\otimes\id_U)\circ(\id_V\otimes \check{R}_{U, W})\circ (\check{R}_{U, V}\otimes\id_W)
\\
=(\id_W\otimes \check{R}_{U, V})\circ(\check{R}_{U, W}\otimes\id_V)\circ(\id_U\otimes \check{R}_{V, W}).
\end{aligned}
\end{eqnarray}
This is the celebrated Yang-Baxter equation. 

\item
The following relation holds in $\End_{\U_q(\fg)}(U\ot V\ot W)$:
\be\label{eq:brb}
\begin{aligned}
\left((\check R_{V,U}^2)\ot\id_W\right)\circ \left(\id_U\ot\check R_{W,V}\right)\circ \left((\check R_{W,U}^2)\ot\id_V\right)
\circ\left(\id_U\ot\check R_{V,W}\right)=&\\
\left(\id_U\ot\check R_{V,W}\right)\circ\left((\check R_{W,U}^2)\ot\id_V\right)\circ
 \left(\id_U\ot\check R_{W,V}\right)\circ \left((\check R_{V,U}^2)\ot\id_W\right).&
\end{aligned}
\ee
\end{enumerate}
\end{proposition}\begin{proof}
Part (1) is well known. Part (2) follows from part (1) and Lemma \ref{lem:aff-fin}. Note that \eqref{eq:YBE} and \eqref{eq:brb} are
essentially braid relations of type $A$ and $B$ respectively.
\end{proof}

\subsection{Representations of the group algebra of $\Gamma_r$ and quotient algebras}
We continue to assume that, as in the previous section,  $\fg$ is a finite dimensional Lie algebra and that $\U_q(\fg)$
is its quantised enveloping algebra. 

\begin{lemma}\label{lem:braid-B-rep}
Let $M, V$ be left $\U_q(\fg)$-modules of type $\bf 1$ with $\dim_{\CK_0} V<\infty$. 
Then
$M\otimes V^{\otimes r}$ is naturally a left $\CG_r$-module, where the corresponding left representation is given by 
\[
\begin{aligned}
 &\nu_r^a:  \CG_r\longrightarrow \End_{\U_q(\fg)}(M\otimes V^{\otimes r}), \text{ where }\\
&\nu_r^a(\xi_1):=\check{R}_{V, M}\check{R}_{M, V}\otimes \id_V^{\otimes (r-1)}=\check{R}_{V, M}^2\otimes \id_V^{\otimes (r-1)}, \\
&\nu_r^a(\sigma_i):=\id_M\otimes \id_V^{\otimes (i-1)}\otimes\check{R}_{V, V}\otimes \id_V^{\otimes (r-i-1)},\quad 1\le i\le r-1.
\end{aligned}
\]
\end{lemma}
\begin{proof} 
The fact that this defines an action of $\Gamma_r$  follows from Proposition \ref{prop:braid} (1) and (2).
\end{proof}

The representation $\nu_r^a$ of $\CG_r$ in Lemma \ref{lem:braid-B-rep} may be modified to obtain another representation
$\tilde\nu_r^a:  \CG_r\longrightarrow \End_{\U_q(\fg)}(M\otimes V^{\otimes r})$ with 
\begin{eqnarray}\label{eq:modified}
\begin{aligned}
&\tilde\nu_r^a(\xi_1):=q^{\frac{2}{k}}\check{R}_{V, M}\check{R}_{M, V}\otimes \id_V^{\otimes (r-1)}, \\
&\tilde\nu_r^a(\sigma_i):=\id_M\otimes \id_V^{\otimes (i-1)}\otimes q^{\frac{1}{k}}\check{R}_{V, V}\otimes \id_V^{\otimes (r-i-1)},\quad 1\le i\le r-1.
\end{aligned}
\end{eqnarray}
In the case $\fg=\fsl_k$, it is convenient to use $\tilde\nu_r^a$, because the eigenvalues of the $\tilde\nu_r^a(\sigma_i)$ 
on the tensor space are just $q$ and $-q\inv$.

The construction in Lemma \ref{lem:braid-B-rep} defines some representations of $\CG_r$ as endomorphisms of  certain
vector spaces. For certain finite dimensional highest weight modules $V$ and for any $M$, these representations 
factor through some of the well known algebras discussed in Section \ref{sect:ATL-algebra}. The next result
collects several such cases.

\begin{theorem}[\cite{DRV}] \label{thm:aff-Hecke-reps}
Maintain the setting of Lemma \ref{lem:braid-B-rep}. 
\begin{enumerate}
\item If $\U_q (\fg)=\U_q (\gl_k)$ (resp. $\U_q(\fsl_k)$) with $k\ge 3$ and $V=\CK_0^k$ is the natural module, 
 the  representation $\nu_r^a$ (resp. $\tilde\nu_r^a$) of $\CG_r$ factors through the affine Hecke algebra $\HH(q)$
 for any $M$. 

\item   If $\U_q (\fg)=\U_q (\gl_2)$ (resp. $\U_q(\fsl_2)$) and $V=\CK_0^2$,  then for any $M$, the representation $\nu_r^a$
(resp. $\tilde\nu_r^a$) factors through the affine Temperley-Lieb algebra $\HTL(q)$.

\item If $\U_q (\fg)$ is $\U_q (\mathfrak{o}_k)$,  $\U_q(\mathfrak{so}_k)$ for $k\ge 3$, or $\U_q (\mathfrak{sp}_k)$ for 
even $k\ge 4$, and $V=\CK_0^k$ is the natural module, 
then for any $M$, the  representation $\nu_r^a$ of $\CG_r$ factors through the affine BMW algebra $\HBMW_r(q-q^{-1},  p^{1-k})$,
where $p= -q^{-1}$ for  $\U_q (\mathfrak{sp}_k)$, and $p=q$ for $\U_q (\mathfrak{o}_k)$  and $\U_q(\mathfrak{so}_k)$.
\end{enumerate}
\end{theorem}
It is an immediate consequence of the Yang-Baxter equation that 
\begin{lemma}\label{lem:braid-rep}
Maintain the setting of Lemma \ref{lem:braid-B-rep}. Let $M=\CK_0$, the $1$-dimensional $\U_q(\fg)$-module.  Then
$V^{\otimes r}$ can be endowed with a left $\CB_r$-module structure with the corresponding left representation 
$\nu_r:  \CB_r\longrightarrow \End_{\U_q(\fg)}(V^{\otimes r})$ given by 
\be\label{eq:nu}
\nu_r(b_i):=\id_V^{\otimes (i-1)}\otimes\check{R}_{V, V}\otimes \id_V^{\otimes (r-i-1)},\quad i=1, 2, \dots, r-1.
\ee
\end{lemma}

For the strongly multiplicity free $\U_q(\fg)$-modules \cite{LZ06},  we have the following result. 
\begin{theorem}[\cite{LZ06}]\label{thm:sw-smf}
Let $V$ denote the natural module for $\U_q (\fg)=$ $\U_q (\gl_k)$,  $\U_q (\fsl_k)$ ($k>2$),  $\U_q (\mathfrak{o}_k)$,  
$\U_q(\mathfrak{so}_{2k+1})$, or $\U_q (\mathfrak{sp}_{2k})$,  the $7$-dimensional
irreducible module for $\U_q(G_2)$, or the $\ell$-dimensional simple modue for $\U_q(\gl_2)$ and $\U_q(\fsl_2)$ for $\ell>2$. 
Then for any $r\in\Z_{>0}$, 
\[
\End_{\U_q(\fg)}\left(V^{\otimes r}\right) =\nu_r(\CB_r).
\]
\end{theorem}

\begin{remark} Note that the role of the affine Hecke algebra and related algebras in the above is very different from that  investigated by \cite{CP} in the context of the representation theory of quantum affine algebras.
\end{remark}

\section{A categorical approach to the affine Temperley-Lieb algebra}\label{sect:cats}
\subsection{A restricted category of tangle diagrams}\label{sect:cat-RT}
Certain categories of tangles, both oriented and unoriented, \cite{FY} and of coloured ribbon graphs \cite{RT}, 
have played  important roles in the construction of link and $3$-manifold invariants. 
We introduce here a category of unoriented tangles up to regular isotopy in the sense of \cite{FY}, which we denote by $\RTC$. 
Our category is a subcategory of the category $S-\mathbb{RT}\mathbb{ang}$, where $S$ is the set $\cC:=\{m,v\}$
which is defined in \cite[Def. 3.1]{FY}. We shall use the language of tangle diagrams, although an equivalent formulation could use the language of coloured ribbon graphs. 
 
The objects of $\RTC$ are sequences of elements of $\cC=\{m, v\}$, which are called ``colours'',  where the empty sequence is allowed. 
The morphisms are spanned, over a base ring $R$, by unoriented tangles up to regular isotopy
in the terminology of \cite{FY}, coloured by $\cC$. Such tangles are represented in ``tangle diagrams'' as unions of arcs; we say that an arc 
is {\em horizontal} if  its end points are either both at the top or at the bottom of the tangle. An arc is
{\em vertical} if it has end points and is not horizontal. To define (the subcategory) $\RTC$, we impose the following two conditions on morphisms.
\begin{enumerate}
\item Closed arcs (i.e. those with no end points) are all coloured by $v$. 

\item Any arc coloured by $m\in\cC$ is vertical, and no two such arcs cross.
\end{enumerate}

The composition of morphisms is explained in \cite{FY} and is essentially by concatenation of tangle diagrams. 

We call $\RTC$ the restricted coloured tangle category.

Since the two ends of any arc have the same colour, we will write $m$ or $v$ beside 
an arc to indicate its colour.   The term ``tangle diagram'' will be abbreviated to ``diagram'' below.

\begin{theorem} \label{thm:tensor-cat}
The category $\RTC$ has the following properties. 
\begin{enumerate} 
\item There is a bi-functor $\otimes: \RTC\times \RTC\longrightarrow \RTC$, called the tensor product, which is defined as follows. 
For any pair of objects $A=(a_1, \dots, a_r)$ and $B=(b_1, \dots, b_s)$, we have
$
A\otimes B=(A, B) := (a_1, \dots, a_r, b_1, \dots, b_s).
$
The tensor product is bilinear on morphisms. Given diagrams $D$ and $D'$, 
$D\otimes D'$ is their juxtaposition with $D$ on the left.

\item The morphisms are generated by the following elementary diagrams under tensor product and composition, 
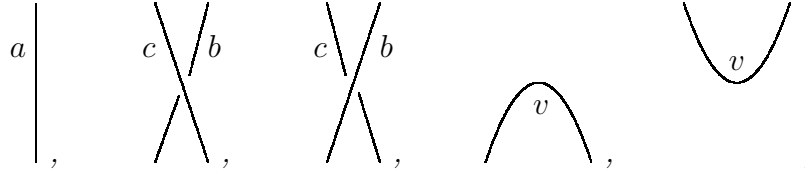
\[
\begin{picture}(20, 70)(0,0)
\put(0, 0){\line(0, 1){60}}
\put(5, 0){,}
\put(-10, 40){$a$}
\end{picture}
\begin{picture}(40, 70)(-60,0)
\qbezier(-35, 60)(-35, 60)(-15, 0)
\qbezier(-15, 60)(-15, 60)(-22, 33)
\qbezier(-35, 0)(-35, 0)(-26, 25)
\put(-10, 0){, }
\put(-15, 40){$b$ }
\put(-40, 40){$c$ }
\end{picture}
\begin{picture}(40, 70)(-140,0)
\qbezier(-90, 0)(-90, 0)(-70, 60)
\qbezier(-90, 60)(-90, 60)(-83, 33)
\qbezier(-70, 0)(-70, 0)(-78, 26)
\put(-65, 0){, }
\put(-70, 40){$b$ }
\put(-95, 40){$c$ }
\end{picture}
\begin{picture}(140, 70)(-50,0)
\qbezier(20, 0)(40, 60)(60, 0)
\put(65, 0){, }
\put(38, 18){$v$}

\qbezier(95, 60)(115, 0)(135, 60)
\put(140, 0){, }
\put(112, 35){$v$}
\end{picture}
\]
where $a, b, c\in \cC=\{m, v\}$ with $b\ne c$ or $b=c=v$. 

\item The defining relations among the above generators are as follows.
\begin{enumerate}
\item Over and under crossings are inverses of each other: 
for all $a, b$ such that either $a\ne b$ or $a=b=v$, 
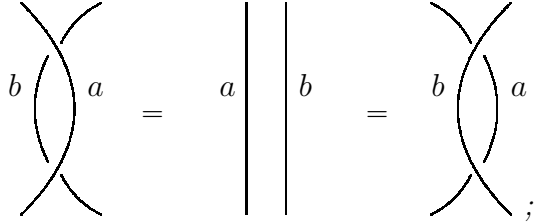
\[
\begin{picture}(70, 90)(65, 0)
\qbezier(60, 0)(100, 40)(60, 80)
\qbezier(70, 20)(60, 40)(70, 60)
\qbezier(75, 15)(80, 5)(90, 0)
\qbezier(75, 65)(80, 75)(90, 80)
\put(55, 45){$b$}
\put(85, 45){$a$}
\put(105, 35){$=$}
\end{picture}
\begin{picture}(50, 90)(110, 0)
\put(120, 0){\line(0, 1){80}}
\put(135, 0){\line(0, 1){80}}
\put(110, 45){$a$}
\put(140, 45){$b$}
\put(165, 35){$=$}
\end{picture}
\begin{picture}(50, 90)(140, 0)
\qbezier(200, 0)(160, 40)(200, 80)
\qbezier(190, 20)(200, 40)(190, 60)
\qbezier(185, 15)(180, 5)(170, 0)
\qbezier(185, 65)(180, 75)(170, 80)
\put(170, 45){$b$}
\put(200, 45){$a$}
\put(205, 0){;}
\end{picture}
\]

\item Braid relation:  for all $a, b, c$ at most one of which is equal to $m$,

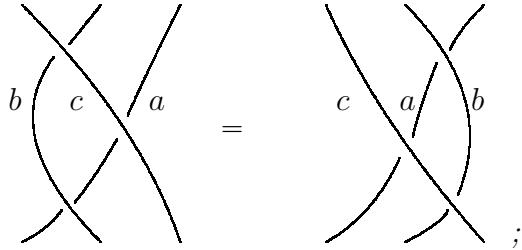
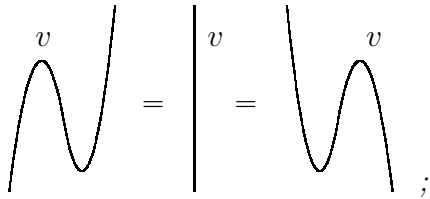
\[
\begin{picture}(85, 100)(0, 0)
\qbezier(30, 90)(30, 90)(18, 75)
\qbezier(12, 70)(-10, 40)(30, 0)
\qbezier(0, 90)(45, 45)(60, 0)
\qbezier(60, 90)(50, 70)(40, 48)
\qbezier(36, 39)(28,25)(20, 15)
\qbezier(15, 12)(10, 5)(0,0)
\put(-5, 50){$b$}
\put(18, 50){$c$}
\put(48, 50){$a$}
\put(75, 40){$=$}
\end{picture}
\begin{picture}(85, 100)(-30, 0)
\qbezier(0, 90)(20, 45)(60, 0)
\qbezier(30, 90)(65, 55)(50, 18)
\qbezier(46, 12)(45, 8)(30, 0)

\qbezier(60, 90)(50, 80)(47, 73)
\qbezier(42, 68)(35, 50)(33, 40)
\qbezier(28, 32)(18, 10)(0, 0)
\put(4, 50){$c$}
\put(28, 50){$a$}
\put(55, 50){$b$}
\put(70, 0){;}
\end{picture}
\]

\item Straightening relations:
\[
\begin{picture}(150, 80)(0,0)

\qbezier(0, 0)(10, 80)(20, 30)
\qbezier(20, 30)(30, -30)(40, 70)
\put(50, 30){$=$}

\put(10, 55){$v$}


\put(70, 0){\line(0, 1){70}}

\put(85, 30){$=$}
\put(75, 55){$v$}


\qbezier(105, 70)(115, -30)(125, 30)
\qbezier(125, 30)(135, 80)(145, 0)
\put(135, 55){$v$}
\put(155, 0){;}
\end{picture}
\]

\item Sliding relations: for all $a\in\cC$, 

\[
\begin{picture}(80, 60)(0,0)
\qbezier(0, 45)(0, 45)(22, 0)
\qbezier(0, 0)(10, 13)(12, 15)
\qbezier(17, 21)(35, 50)(60, 0)
\put(-8, 35){$a$}
\put(32, 18){$v$}
\put(75, 15){$=$}
\end{picture}
\begin{picture}(80, 60)(-20,0)
\qbezier(0, 0)(25, 50)(45, 20)
\qbezier(49, 15)(50, 15)(60, 0)
\qbezier(60, 45)(60, 45)(38, 0)
\put(62, 35){$a$}
\put(22, 18){$v$}
\put(70, 0){;}
\end{picture}
\]
\[
\begin{picture}(80, 60)(0,0)
\qbezier(0, 0)(30, 60)(60, 0)
\qbezier(0, 45)(0, 45)(10, 22)
\qbezier(15, 15)(15, 15)(22, 0)
\put(-8, 35){$a$}
\put(32, 18){$v$}
\put(75, 15){$=$}
\end{picture}
\begin{picture}(80, 60)(-20,0)
\qbezier(0, 0)(30, 60)(60, 0)
\qbezier(60, 45)(60, 45)(50, 22)
\qbezier(45, 15)(45, 15)(38, 0)
\put(62, 35){$a$}
\put(22, 18){$v$}
\put(70, 0){;}
\end{picture}
\]
 
 \item Twists, 
 \[
 \begin{picture}(80, 130)(0,0)
 \qbezier(0, 100)(-5, 110)(-5, 120)
\qbezier(0, 100)(0, 100)(20, 70)
\qbezier(0, 70)(0, 70)(8, 82)
\qbezier(20, 100)(20, 100)(13, 88)
\qbezier(20, 100)(30, 110)(32, 85)
\qbezier(20, 70)(30, 60)(32, 85)

\qbezier(00, 70)(-5, 60)(00, 50)
\put(-15, 60){$v$}

\qbezier(20, 50)(0, 20)(0, 20)
\qbezier(0, 50)(0, 50)(8, 38)
\qbezier(12, 33)(20, 20)(20, 20)
\qbezier(20, 50)(30, 60)(32, 35)
\qbezier(20, 20)(30, 10)(32, 35)
\qbezier(-5, 0)(-5, 10)(0, 20)

\put(55, 60){$=$}
\end{picture}
\begin{picture}(20, 130)(-15,20)
\put(0, 20){\line(0, 1){120}}
\put(5, 80){$v$}
\put(10, 20){.}
\end{picture}
 \]
\end{enumerate}
\end{enumerate}
\end{theorem}
\begin{proof} This was proved in \cite[Theorem 3.5]{FY} and in \cite[Lemma 5.3]{RT}.  
The result of \cite[Theorem 3.5]{FY} in fact applies to several categories of tangles; the case which covers our theorem  is that of $S-\mathbb{RT}\mathbb{ang}$ with $S=\cC$. Note that \cite[Theorem 3.5]{FY} does not involve colours, but this is not an issue 
as colours merely label components of tangles.  We can also extract our theorem from \cite[Lemma 5.3]{RT} by removing directions of ribbon graphs and forbidding coupons. There is also a direct proof along the lines of \cite[Appendix]{LZ14}.
\end{proof}

\noindent
{\bf Another type of picture}. A second way of representing the category $\RTC$ is as follows.
We depict arcs coloured by $m$ as thick arcs called poles, and arcs coloured by $v$ as thin arcs. 
This way a diagram automatically carries the information about the colours of its arcs, so that we may drop the 
letters for colours from the diagram.  For example, we have the following diagram $(m, v^2, m, v^2)\to (v, m, v, m, v^2)$.
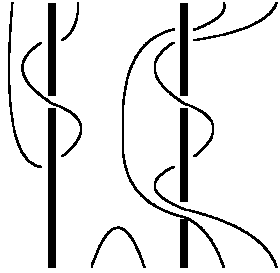
\[
\begin{picture}(120, 100)(-90,0)
{
\linethickness{1mm}
\put(-15, 65){\line(0, 1){35}}
\put(-15, 0){\line(0, 1){60}}
}
\qbezier(-19, 38)(-35, 40)(-30, 100)
\qbezier(-15, 62)(5, 55)(-11, 42)
\qbezier(-15, 62)(-35, 75)(-19, 85)
\qbezier(-11, 86)(-5, 90)(-5, 100)
\qbezier(0, 0)(10, 30)(20, 0)
\end{picture}
\begin{picture}(150, 100)(-20,0)
{
\linethickness{1mm}
\put(-15, 65){\line(0, 1){35}}
\put(-15, 25){\line(0, 1){35}}
\put(-15, 0){\line(0, 1){18}}
}

\qbezier(-19, 38)(-35, 30)(-15, 22)
\qbezier(-15, 62)(5, 55)(-11, 42)
\qbezier(-15, 62)(-35, 75)(-19, 85)
\qbezier(20, 100)(16, 90)(-11, 86)
\qbezier(0, 100)(2, 95)(-11, 90)

\qbezier(-19, 90)(-40, 85)(-38, 50)
\qbezier(-15, 19)(-40, 25)(-38, 50)
\qbezier(-15, 19)(-5, 15)(0, 0)
\qbezier(-15, 22)(15, 15)(20, 0)
\end{picture}
\]
We recover the diagramatics for $B_r$ and $\Gamma_r$ given in Section \ref{sect:diagrams} 
by identifying thin arcs with strings and thick arcs with poles.

\begin{remark}
Given any $A\in\cC^N$, we denote by $\Gamma(A)$ the set of diagrams $A\to  A$ with vertical arcs only. 
Then $\Gamma(A)$ forms a group. In particular  $\Gamma(v^r)\cong B_r$ is the braid group of type $A$, 
and $\Gamma(m, v^r)\cong \Gamma_r$ is the braid group of type $B$.  We will call $\Gamma(A)$ a {\em multi-polar braid group} 
if $A$ has more than one $m$ entry. 
\end{remark}

\subsection{The affine Temperley-Lieb category}
%
%
%
%
%
%

We now introduce a quotient category of $\RTC$ denoted by $\ATLC(q)$, which we refer to as the affine Temperley-Lieb category. 
The objects of this category are the same as those of $\RTC$. Given any two objects $T$ and $B$ in $\ATLC(q)$, 
$\Hom_{\ATLC(q)}(B, T)$ is the quotient space of $\Hom_{\RTC}(B, T)$ obtained by imposing locally the following 
skein relations and free loop removal. 

\noindent Skein relations:
\be\label{eq:skein1}
\begin{aligned}
\begin{picture}(150, 70)(0,0)
\put(-15, 30){$q^{\frac{1}{2}}$}
\qbezier(0, 60)(0, 60)(20, 0)
\qbezier(20, 60)(20, 60)(13, 33)
\qbezier(0, 0)(0, 0)(8, 24)
\put(30, 30){$=$}
\put(50, 30){$q$}
\put(60, 0){\line(0, 1){60}}
\put(80, 0){\line(0, 1){60}}

\put(95, 30){$+$}

\qbezier(120, 0)(135, 50)(150, 0)
\qbezier(120, 60)(135, 10)(150,60)
\put(155, 0){, }
\end{picture}
\end{aligned}
\ee

\be\label{eq:skein2}
\begin{aligned}
\begin{picture}(150, 70)(0,0)
\put(-20, 30){$q^{-\frac{1}{2}}$}
\qbezier(0, 0)(0, 0)(20, 60)
\qbezier(0, 60)(0, 60)(7, 33)
\qbezier(20, 0)(20, 0)(12, 26)

\put(30, 30){$=$}
\put(40, 30){$q^{-1}$}

\put(60, 0){\line(0, 1){60}}
\put(80, 0){\line(0, 1){60}}

\put(95, 30){$+$}

\qbezier(120, 0)(135, 50)(150, 0)
\qbezier(120, 60)(135, 10)(150,60)
\put(155, 0){;}
\end{picture}
\end{aligned}
\ee

\noindent
Free loop removal:
\be\label{eq:flr}
\begin{aligned}
\begin{picture}(100, 60)(20,0)
\put(0, 30){\circle{25}}
\put(35, 25){$=$}
\put(62, 25){$- (q+q^{-1})$.}
\end{picture}
\end{aligned}
\ee

The image in $\ATLC(q)$ of a diagram in $\RTC$ will be  depicted by the same graph, but is understood to obey the above relations.  
The composition of morphisms  in $\ATLC(q)$ is that inherited from $\RTC$. 

We denote by $\TLC(q)$ the full subcategory of $\ATLC(q)$ with objects the sequences
 in which only the symbol $v$ occurs. This will be referred to as the {\em \tl category}.  

\begin{remark} \begin{enumerate}
\item The reason for taking the skein relations with $q^{\pm\frac{1}{2}}$ here, and also in the definition of the 
Temperley-Lieb algebra of type $B$  in Section \ref{sect:TLBC} below  will become clear from Lemma \ref{rem:norm-factors}.
\item The affine Temperley-Lieb category we use here is closely related to that in \cite{GL98}. It is in effect an ``extended'' version
of the category $T^a(q)$ of {\it op. cit.}.
\end{enumerate}
\end{remark}

Consider the morphisms from $m$ to $m$ depicted in \eqref{eq:z12}. 
\begin{eqnarray}\label{eq:z12}
\begin{aligned}
\begin{picture}(160, 70)(-70,0)
\put(-70, 48){$z_1:=$}
{
\linethickness{1mm}
\put(-15, 45){\line(0, 1){55}}
\put(-15, 0){\line(0, 1){40}}
}

\qbezier(-15, 42)(15, 50)(-11, 65)
\qbezier(-15, 42)(-45, 50)(-19, 65)
\put(5, 0){, }
\end{picture}
\begin{picture}(70, 100)(-40,0)
\put(-55, 48){$z_2:=$}
{
\linethickness{1mm}
\put(-15, 65){\line(0, 1){35}}
\put(-15, 25){\line(0, 1){35}}
\put(-15, 0){\line(0, 1){18}}
}

\qbezier(-19, 38)(-35, 30)(-15, 22)
\qbezier(-15, 62)(5, 55)(-11, 42)
\qbezier(-15, 62)(-35, 75)(-19, 85)

\qbezier(-11, 88)(15, 95)(10, 50)
\qbezier(-15, 22)(0, 15)(10, 50)
\put(5, 0){. }
\end{picture}
\end{aligned}
\end{eqnarray}
The following result is clear diagrammatically. 
\begin{lemma} \label{lem:central-ATL} The elements $z_1$ and $z_2$ are central in $\ATLC(q)$ in the following sense.  
Let $D$ be a diagram 
$D: (A_1, m, A_2)\to (B_1, m, B_2)$ in $\ATLC(q)$, where the two $m$'s shown are connected by a thick arc. Then
\[
D(\id_{A_1}\otimes z_i \otimes \id_{A_2}) =  (\id_{B_1}\otimes z_i \otimes \id_{B_2}) D, \quad i=1, 2.
\]
\end{lemma}

\noindent{\bf Notation}.
Henceforth, for any invertible scalar $t\in R$ (the ground ring) we write $\delta_t :=-(t +t^{-1})$ , and 
for the special case $t=q$, set $\delta:=\delta_q=- (q+q^{-1})$.

\begin{lemma}\label{lem:central-skein-ATL} The following relation holds in $\ATLC(q)$.
\begin{eqnarray*}
\phantom{XXX}
\begin{picture}(110, 100)(-50,0)
\put(-40, 45){$q \delta$}
{
\linethickness{1mm}
\put(-15, 65){\line(0, 1){35}}
\put(-15, 25){\line(0, 1){35}}
\put(-15, 0){\line(0, 1){18}}
}

\qbezier(-19, 38)(-35, 30)(-15, 22)
\qbezier(-15, 62)(5, 55)(-11, 42)
\qbezier(-15, 62)(-35, 75)(-19, 85)
\qbezier(0, 100)(0, 90)(-11, 86)
\qbezier(-15, 22)(0, 15)(0, 0)
\end{picture}
\begin{picture}(120, 100)(-30,0)
\put(-70, 45){$+\ \ \delta z_1$}
{
\linethickness{1mm}
\put(-15, 0){\line(0, 1){24}}
\put(-15, 30){\line(0, 1){70}}
}
\qbezier(-18, 60)(-40, 40)(-15, 27)
\qbezier(-15, 27)(0, 20)(0, 0)
\qbezier(0, 100)(0, 67)(-12, 64)
\end{picture}
\begin{picture}(50, 100)(-15,0)
\put(-80, 45){$=\ (q z_2+z_1^2)$}
{
\linethickness{1mm}
\put(-5, 0){\line(0, 1){100}}
}
\put(15, 0){\line(0, 1){100}}
\put(20, 0){,}
\end{picture}
\end{eqnarray*}
where the right hand side should be understood as $(q z_2+z_1^2)\otimes\id_v$.
\end{lemma}
\begin{proof} We prove the relation diagrammatically.   Let 
\[
\begin{picture}(120, 100)(-30,0)
\put(-55, 45){$\hat{x}_1=$}
{
\linethickness{1mm}
\put(-15, 0){\line(0, 1){25}}
\put(-15, 30){\line(0, 1){60}}
}
\qbezier(-18, 60)(-35, 45)(-15, 28)
\qbezier(-15, 28)(0, 16)(0, 0)
\qbezier(0, 90)(0, 67)(-12, 64)
\put(10, 0){\line(0, 1){90}}
\put(15, 0){, }
\end{picture}
\begin{picture}(80, 100)(-30,0)
\put(-55, 45){$\hat{x}_2=$}
{
\linethickness{1mm}
\put(-15, 0){\line(0, 1){25}}
\put(-15, 30){\line(0, 1){60}}
}
\put(-5, 0){\line(0, 1){17}}
\put(-5, 22){\line(0, 1){68}}
\put(-5, 29){\line(0, 1){61}}
\qbezier(-18, 60)(-35, 30)(0, 18)
\qbezier(0, 18)(8, 15)(10, 0)
\qbezier(10, 90)(10, 72)(-3, 70)
\qbezier(-7, 68)(-8, 67)(-12, 64)
\put(15, 0){.}
\end{picture}
\]
Then $\hat{x}_1 \hat{x}_2= \hat{x}_2 \hat{x}_1$.  Multiplying this by $q$ and using the skein relation, we obtain  
\[
\begin{picture}(120, 100)(-90,0)
\put(-40, 48){$q$}
{
\linethickness{1mm}
\put(-15, 65){\line(0, 1){35}}
\put(-15, 25){\line(0, 1){35}}
\put(-15, 0){\line(0, 1){18}}
}

\qbezier(-19, 38)(-35, 30)(-15, 22)
\qbezier(-15, 62)(5, 55)(-11, 42)
\qbezier(-15, 62)(-35, 75)(-19, 85)
\qbezier(-11, 86)(-5, 90)(-5, 100)
\qbezier(-15, 22)(10, 15)(10, 100)
\qbezier(-5, 0)(3, 35)(10, 0)
\end{picture}
\begin{picture}(120, 100)(-45,0)
\put(-50, 48){$+$}
{
\linethickness{1mm}
\put(-15, 65){\line(0, 1){35}}
\put(-15, 25){\line(0, 1){35}}
\put(-15, 0){\line(0, 1){18}}
}

\qbezier(-19, 48)(-50, 30)(-15, 20)
\qbezier(-15, 20)(25, 30)(-11, 48)
\qbezier(-15, 62)(5, 55)(10, 100)
\qbezier(-15, 62)(-35, 75)(-19, 85)

\qbezier(-5, 0)(3, 35)(10, 0)
\qbezier(-11, 86)(-5, 90)(-5, 100)
\end{picture}
\begin{picture}(50, 110)(-10,0)
\put(-60, 48){$=\  \ q $}
{
\linethickness{1mm}
\put(-15, 65){\line(0, 1){45}}
\put(-15, 25){\line(0, 1){35}}
\put(-15, 0){\line(0, 1){18}}
}

\qbezier(-19, 38)(-35, 30)(-15, 22)
\qbezier(-15, 62)(5, 55)(-11, 42)
\qbezier(-15, 62)(-35, 75)(-19, 85)
\qbezier(-5, 110)(3, 80)(10, 110)
\qbezier(-11, 88)(15, 95)(10, 50)
\qbezier(10, 0)(10, 15)(10, 50)

\qbezier(-15, 22)(-5, 15)(-5, 0)
\end{picture}
\begin{picture}(150, 70)(-50,0)
\put(-60, 45){$+$}
{
\linethickness{1mm}
\put(-15, 65){\line(0, 1){35}}
\put(-15, 25){\line(0, 1){35}}
\put(-15, 0){\line(0, 1){18}}
}

\qbezier(-19, 45)(-45, 30)(-15, 20)

\qbezier(-11, 47)(8, 60)(10, 0)

\qbezier(-15, 20)(-5, 20)(-5, 0)
\qbezier(-15, 62)(15, 70)(-11, 85)
\qbezier(-15, 62)(-45, 70)(-19, 85)
\qbezier(-5, 100)(3, 70)(10, 100)
\put(15, 0){.}
\end{picture}
\]
Composing this equation with 
$
\begin{picture}(50, 50)(-10,20)
{
\linethickness{1mm}
\put(0, 10){\line(0, 1){30}}
}

\qbezier(10, 10)(15, 40)(20, 10)
\end{picture}
$
on the top,  and then pulling  the right end point \\

\vspace{1mm}
\noindent 
of the string in the resulting diagram to the top, we obtain the stated relation.
%
%
%
\end{proof}

\begin{definition}\label{def:ext-ATL}
For any integer $r\geq 0$, the space $\Hom_{\ATLC(q)}((m, v^r), (m, v^r))$ is an associative algebra with 
multiplication given by the composition of morphisms.  We denote this algebra by $\HTL^{ext}_r(q)$, and 
call it the extended affine Temperley-Lieb algebra. 
\end{definition}

Observe that the algebra $\HTL^{ext}_r(q)$ defined above
contains the elements $z_1\otimes\id_{v^r},  z_2\otimes\id_{v^r}$ depicted in \eqref{eq:central-zs}, 
which are central by Lemma \ref{lem:central-ATL}. 
\begin{eqnarray}\label{eq:central-zs}
\begin{aligned}
\begin{picture}(200, 70)(-70,0)
\put(-100, 48){$z_1\otimes\id_{v^r}:=$}
{
\linethickness{1mm}
\put(-15, 45){\line(0, 1){55}}
\put(-15, 0){\line(0, 1){40}}
}

\qbezier(-15, 42)(15, 50)(-11, 65)
\qbezier(-15, 42)(-45, 50)(-19, 65)


\put(15, 0){\line(0, 1){100}}
\put(35, 0){\line(0, 1){100}}
\put(20, 45){...}
\put(20, 10){$r$}
\put(40, 0){, }
\end{picture}
\begin{picture}(70, 110)(-40,0)
\put(-85, 48){$z_2\otimes\id_{v^r}:=$}
{
\linethickness{1mm}
\put(-15, 65){\line(0, 1){45}}
\put(-15, 25){\line(0, 1){35}}
\put(-15, 0){\line(0, 1){18}}
}

\qbezier(-19, 38)(-35, 30)(-15, 22)
\qbezier(-15, 62)(5, 55)(-11, 42)
\qbezier(-15, 62)(-35, 75)(-19, 85)
\put(20, 0){\line(0, 1){110}}

\qbezier(-11, 88)(15, 95)(10, 50)
\qbezier(-15, 22)(0, 15)(10, 50)
\put(25, 45){...}
\put(25, 10){$r$}
\put(40, 0){\line(0, 1){110}}
\put(45, 0){. }
\end{picture}
\end{aligned}
\end{eqnarray}

The following fact is an obvious consequence of Lemma \ref{lem:central-skein-ATL}. 
\begin{lemma} \label{lem:central-skein}
The following relation holds in $\HTL^{ext}_r(q)$.
\begin{eqnarray*}
\phantom{XXX}
\begin{picture}(150, 100)(-50,0)
\put(-40, 45){$q \delta$}
{
\linethickness{1mm}
\put(-15, 65){\line(0, 1){35}}
\put(-15, 25){\line(0, 1){35}}
\put(-15, 0){\line(0, 1){18}}
}

\qbezier(-19, 38)(-35, 30)(-15, 22)
\qbezier(-15, 62)(5, 55)(-11, 42)
\qbezier(-15, 62)(-35, 75)(-19, 85)
\qbezier(0, 100)(0, 90)(-11, 86)
\qbezier(-15, 22)(0, 15)(0, 0)
\put(15, 0){\line(0, 1){100}}
\put(20, 45){...}
\put(35, 0){\line(0, 1){100}}
\put(55, 45){$+$}
\end{picture}
\begin{picture}(150, 100)(-30,0)
\put(-50, 45){$\delta z_1$}
{
\linethickness{1mm}
\put(-15, 0){\line(0, 1){24}}
\put(-15, 30){\line(0, 1){70}}
}
\qbezier(-18, 60)(-40, 40)(-15, 27)
\qbezier(-15, 27)(0, 20)(0, 0)
\qbezier(0, 100)(0, 67)(-12, 64)
\put(15, 0){\line(0, 1){100}}
\put(35, 0){\line(0, 1){100}}
\put(20, 45){...}
\end{picture}
\begin{picture}(150, 100)(-15,0)
\put(-80, 45){$=\ (q z_2+z_1^2)$}
{
\linethickness{1mm}
\put(-5, 0){\line(0, 1){100}}
}
\put(15, 0){\line(0, 1){100}}
\put(35, 0){\line(0, 1){100}}
\put(20, 45){...}
\put(40, 0){,}
\end{picture}
\end{eqnarray*}
where the right hand side should be understood as $(q z_2+z_1^2)\otimes\id_{v^r}$. 
\end{lemma}

Recall the affine Temperley-Lieb algebra $\HTL_r(q)$ which was defined in \eqref{eq:defhtl}. The next statement is easily verified
by checking that the relations  \eqref{eq:extra-inv} which hold in $\HTL_r(q)$ are preserved by $\wp$ using  Lemma \ref{lem:central-skein}.

\begin{lemma} There is an injection $\wp: \HTL_r(q)\to\HTL^{ext}_r(q)$ of algebras  given by 
\[
\begin{picture}(150, 90)(-15,-5)
\put(-40, 40){$e_i\mapsto$}
{
\linethickness{1mm}
\put(-5, 0){\line(0, 1){80}}
}
\put(15, 0){\line(0, 1){80}}
\put(35, 0){\line(0, 1){80}}
\qbezier(55, 0)(65, 70)(75, 0)
\qbezier(55, 0)(65, 70)(75, 0)
\qbezier(55, 80)(65, 10)(75, 80)
\put(55, -8){$i$}
\put(95, 0){\line(0, 1){80}}
\put(115, 0){\line(0, 1){80}}
\put(20, 45){...}
\put(100, 45){...}
\put(120, 0){,}
\end{picture}
\begin{picture}(150, 90)(-80,-5)
\put(-65, 40){$x_1\mapsto$}
{
\linethickness{1mm}
\put(-15, 0){\line(0, 1){24}}
\put(-15, 30){\line(0, 1){50}}
}
\qbezier(-18, 60)(-40, 40)(-15, 27)
\qbezier(-15, 27)(0, 20)(0, 0)
\qbezier(0, 80)(0, 67)(-12, 64)
\put(15, 0){\line(0, 1){80}}
\put(75, 0){\line(0, 1){80}}
\put(25, 45){............}
\put(80, 0){,}
\end{picture}
\]
for $i=1, 2, \dots, r-1$.
\end{lemma}

\begin{remark}
The extended affine Temperley-Lieb algebra $\HTL^{ext}_r(q)$ is strictly bigger than the affine 
Temperley-Lieb algebra $\HTL_r(q)$, as  $z_1\otimes\id_{v^r}$ and $z_2\otimes\id_{v^r}$ do not belong to $\wp(\HTL_r(q))$. It is easy to see 
that $\HTL^{ext}_r(q)$ is generated by  $z_1\otimes\id_{v^r},   z_2\otimes\id_{v^r}$ and $\wp(\HTL_r(q))$.
\end{remark}

\begin{remark} \label{rem:skein-universal} Lemma \ref{lem:central-skein} may be thought as a ``skein relation'' over the centre of $\HTL_r^{ext}(q)$.  
In any irreducible representation of $\HTL^{ext}_r(q)$, 
the central elements $z_1\otimes\id_{v^r}$ and $z_2\otimes\id_{v^r}$ act as scalars, and the algebra therefore acts through a quotient in which
a skein relation in the usual sense is satisfied. 
\end{remark}

\subsection{A quotient category of the affine Temperley-Lieb category}

We next introduce a two-parameter family of quotients of the affine Temperley-Lieb category $\ATLC(q)$, 
which are obtained by imposing conditions on the removal of tangled loops. 
The morphisms $z_1$ and $z_2$ from $m$ to $m$ in the category $\ATLC(q)$ are as defined in \eqref{eq:z12}.

\begin{definition} 
Given scalars $a_1, a_2$,  let $\ATLC(q, a_1, a_2)$ be the category whose objects are the same as those of $\ATLC(q)$ 
and whose modules of morphisms are obtained from those in $\ATLC(q)$ by 
imposing all relations which are consequences of the following two equalities.
\[
\begin{picture}(160, 70)(0,0)
{
\linethickness{1mm}
\put(-15, 45){\line(0, 1){55}}
\put(-15, 0){\line(0, 1){40}}
}

\qbezier(-15, 42)(15, 50)(-11, 65)
\qbezier(-15, 42)(-45, 50)(-19, 65)
\put(15, 45){$=\ \  a_1$ }
{
\linethickness{1mm}
\put(55, 0){\line(0, 1){100}}
\put(60, 0){,}
}
\end{picture}
\begin{picture}(70, 100)(0,0)
{
\linethickness{1mm}
\put(-15, 65){\line(0, 1){35}}
\put(-15, 25){\line(0, 1){35}}
\put(-15, 0){\line(0, 1){18}}
}

\qbezier(-19, 38)(-35, 30)(-15, 22)
\qbezier(-15, 62)(5, 55)(-11, 42)
\qbezier(-15, 62)(-35, 75)(-19, 85)

\qbezier(-11, 88)(15, 95)(10, 50)
\qbezier(-15, 22)(0, 15)(10, 50)

\put(15, 45){$=\ \  a_2$ }
{
\linethickness{1mm}
\put(55, 0){\line(0, 1){100}}
}
\put(60, 0){. }
\end{picture}
\] 
We refer to $\ATLC(q, a_1, a_2)$ as the {\em multi-polar \tl category with two parameters}.
\end{definition}

\begin{definition}\label{def:tlbc-2p}
Denote by $\TLBC(q, a_1, a_2)$ the full subcategory of $\ATLC(q, a_1, a_2)$ 
with objects of the form $(m, v^r):=(m, \underbrace{v, \dots, v}_r)$ for all $r\in \Z_{\ge 0}$.
This is the {\em two-parameter \tl category of type $B$ }.
 \end{definition}

It is then clear that in the quotient, we have the following relation. 
\begin{lemma}\label{lem:skein-a-a} The following skein relation holds in $\ATLC(q, a_1, a_2)$. 
\begin{eqnarray*}
\phantom{XXX}
\begin{picture}(110, 100)(-50,0)
\put(-40, 45){$q \delta$}
{
\linethickness{1mm}
\put(-15, 65){\line(0, 1){35}}
\put(-15, 25){\line(0, 1){35}}
\put(-15, 0){\line(0, 1){18}}
}

\qbezier(-19, 38)(-35, 30)(-15, 22)
\qbezier(-15, 62)(5, 55)(-11, 42)
\qbezier(-15, 62)(-35, 75)(-19, 85)
\qbezier(0, 100)(0, 90)(-11, 86)
\qbezier(-15, 22)(0, 15)(0, 0)
\end{picture}
\begin{picture}(120, 100)(-30,0)
\put(-70, 45){$+\ \ \delta a_1$}
{
\linethickness{1mm}
\put(-15, 0){\line(0, 1){24}}
\put(-15, 30){\line(0, 1){70}}
}
\qbezier(-18, 60)(-40, 40)(-15, 27)
\qbezier(-15, 27)(0, 20)(0, 0)
\qbezier(0, 100)(0, 67)(-12, 64)
\end{picture}
\begin{picture}(50, 100)(-15,0)
\put(-80, 45){$=\ (q a_2+a_1^2)$}
{
\linethickness{1mm}
\put(-5, 0){\line(0, 1){100}}
}
\put(15, 0){\line(0, 1){100}}
\put(20, 0){,}
\end{picture}
\end{eqnarray*}
\end{lemma}

\begin{definition} For any positive integer $r$, let $\TLB_r(q, a_1, a_2)$ be the associative algebra with underlying 
vector space $\Hom_{\ATLC(q, a_1, a_2)}((m, v^r), (m, v^r))$ and multiplication given by composition of 
morphisms. The algebra $\TLB_r(q, a_1, a_2)$ will be referred to as a two-parameter Temperley-Lieb algebra of type $B$.  
\end{definition}

The following statements are easy consequences of Lemma \ref{lem:central-skein}.  

\begin{lemma} 
\begin{enumerate}
\item  Let $J_{a_1, a_2}$ be the two-sided ideal of $\HTL^{ext}_r(q)$ generated by the elements
 $z_1\otimes\id_{v^r}-a_1$ and $z_2\otimes\id_{v^r}-a_2$ (see \eqref{eq:central-zs}).  Then 
\[
\TLB_r(q, a_1, a_2) = \HTL^{ext}_r(q)/J_{a_1, a_2}.
\]
Write $\rho^{ext}: \HTL^{ext}_r(q) \longrightarrow \TLB_r(q, a_1, a_2) $ for the canonical surjection. 

\item There exists a surjective homomorphism $\rho: \HTL_r(q)\longrightarrow \TLB_r(q, a_1, a_2)$  of algebras defined 
as the the following composition: 
\[ 
\HTL_r(q)\stackrel{\wp}\longrightarrow\HTL^{ext}_r(q)\stackrel{\rho^{ext}}\longrightarrow \TLB_r(q, a_1, a_2). 
\]
\item Each simple finite dimensional $\HTL_r(q)$-module $N$ is a pull-back by $\rho$ of  a simple $\TLB_r(q, a_1, a_2)$-module 
for unique values of $a_1$ and $a_2$ which depend on $N$. 
\end{enumerate}
\end{lemma}

\subsection{The Templey-Lieb category of type $B$}\label{sect:TLBC}
\subsubsection{Definition}
We now introduce another quotient category of $\RTC$, 
denoted  by $\ATLC(q, \Omega)$, which depends on an invertible scalar $\Omega$. 
 Its objects are the same as those of $\RTC$. Given any two objects $T$ and $B$ in $\ATLC(q, \Omega)$, 
the space $\Hom_{\ATLC(q, \Omega)}(B, T)$ of morphisms is the quotient space of $\Hom_{\RTC}(B, T)$ obtained by imposing 
locally the relations listed below and referred to respectively as skein relations, free loop removal and tangled loop removal. 

\noindent Skein relations:
\[
\begin{picture}(150, 70)(0,0)
\put(-15, 30){$q^{\frac{1}{2}}$}
\qbezier(0, 60)(0, 60)(20, 0)
\qbezier(20, 60)(20, 60)(13, 33)
\qbezier(0, 0)(0, 0)(8, 24)
\put(30, 30){$=$}
\put(50, 30){$q$}

\put(60, 0){\line(0, 1){60}}
\put(80, 0){\line(0, 1){60}}

\put(95, 30){$+$}

\qbezier(120, 0)(135, 50)(150, 0)
\qbezier(120, 60)(135, 10)(150,60)
\put(155, 0){, }
\end{picture}
\]

\[
\begin{picture}(150, 70)(0,0)
\put(-20, 30){$q^{-\frac{1}{2}}$}
\qbezier(0, 0)(0, 0)(20, 60)
\qbezier(0, 60)(0, 60)(7, 33)
\qbezier(20, 0)(20, 0)(12, 26)

\put(25, 30){$=q^{-1}$}

\put(60, 0){\line(0, 1){60}}
\put(80, 0){\line(0, 1){60}}

\put(95, 30){$+$}

\qbezier(120, 0)(135, 50)(150, 0)
\qbezier(120, 60)(135, 10)(150,60)
\put(155, 0){, }
\end{picture}
\]

\[
\phantom{XXXXXX}
\begin{picture}(80, 100)(-30,0)
\put(-40, 48){$q$}
{
\linethickness{1mm}
\put(-15, 65){\line(0, 1){35}}
\put(-15, 25){\line(0, 1){35}}
\put(-15, 0){\line(0, 1){18}}
}

\qbezier(-19, 38)(-35, 30)(-15, 22)
\qbezier(-15, 62)(5, 55)(-11, 42)
\qbezier(-15, 62)(-35, 75)(-19, 85)

\qbezier(0, 100)(-6, 90)(-11, 87)
\qbezier(-15, 22)(-5, 15)(0, 0)
\put(10, 45) {$=$}
\end{picture}
\begin{picture}(150, 100)(-80,0)
\put(-95, 48){$(\Omega + \Omega^{-1})$}
{
\linethickness{1mm}
\put(-15, 40){\line(0, 1){60}}
\put(-15, 0){\line(0, 1){35}}
}

\qbezier(-15, 37)(-45, 53)(-19, 70)
\qbezier(0, 100)(0, 80)(-11, 72)
\qbezier(-15, 37)(0, 30)(0, 0)
\end{picture}
\begin{picture}(50, 100)(20,0)
\put(-40, 48){$-\  q^{-1}$}
{
\linethickness{1mm}
\put(5, 0){\line(0, 1){100}}
}

\put(20, 0){\line(0, 1){100}}
\put(25, 0){;}
\end{picture}
\]
\\

\noindent
Free loop removal:
\[
\begin{picture}(100, 60)(20,0)
\put(0, 30){\circle{25}}
\put(35, 25){$=$}
\put(65, 25){$- (q+q^{-1})$;}
\end{picture}
\]

\noindent Tangled loop removal:
\[
\begin{picture}(150, 100)(-90,0)
{
\linethickness{1mm}
\put(-15, 65){\line(0, 1){35}}
\put(-15, 42){\line(0, 1){58}}
\put(-15, 0){\line(0, 1){36}}
}

\qbezier(-19, 68)(-50, 50)(-15, 38)

\qbezier(-11, 68)(-3, 70)(0, 85)

\qbezier(-15, 38)(-5, 38)(0, 20)

\put(20, 20){\line(0, 1){65}}

\qbezier(0, 20)(10, -10 ) (20, 20)
\qbezier(0, 85)(10, 105) (20, 85)

\put(35, 45){$=$}
\end{picture}
\begin{picture}(150, 110)(-50,0)
\put(-50, 45){$-(\Omega+\Omega^{-1})$}
{
\linethickness{1mm}
\put(20, 0){\line(0, 1){100}}
}

\put(25, 0){.}
\end{picture}
\]

The image in $\ATLC(q, \Omega)$ of a tangle diagram from $\RTC$ will be  depicted by the same diagram, but is understood to represent a coset, and therefore obeys the above relations. Composition of morphisms  in $\ATLC(q, \Omega)$ is inherited from that in $\RTC$. 

\begin{definition}\label{def:ATLC-q}
The category $\ATLC(q, \Omega)$ is  the {\em multi-polar \tl category with one parameter}.
\end{definition}

The following lemma is obtained by specialising the parameters of the two parameter multi-polar \tl category. 
\begin{lemma}\label{lem:iso-tlcb-2}
There is an equivalence $\ATLC(q, \Omega)\cong\ATLC(q, a_1,  a_2)$ of categories, where the parameters $a_1, a_2$ are respectively given by
\[
a_1=-(\Omega+\Omega^{-1}), \quad a_2=-q^{-1}((\Omega+\Omega^{-1})^2+\delta q^{-1}). 
\]
\end{lemma}
\begin{proof} 
This is clear. 
\end{proof}

%
%
%
%
It is further clear that the \tl category $\TLC(q)$ is contained in $\ATLC(q, \Omega)$ as a full subcategory. 
There is another full subcategory of $\ATLC(q, \Omega)$, which is of particular interest to us.
 \begin{definition}\label{def:tl(b)}
The Temperley-Lieb category of type $B$ is the full subcategory of $\ATLC(q, \Omega)$ with objects of the 
form $(m, v^r):=(m, \underbrace{v, \dots, v}_r)$ for all $r\in \Z_{\ge 0}$. This category will be denoted by  $\TLBC(q, \Omega)$.  
\end{definition}

\begin{remark}\label{rem:tlsub} Note that the finite \tl category $\TLC(q)$ may be thought of as a subcategory of $\TLBC(q,\Omega)$,
since diagrams from $(v^r)$ to $(v^s)$ are evidently in bijection with diagrams $(m,v^r)$ to $(m,v^s)$ which have no entanglement with the pole.
Thus $\TLC(q)$ is the subcategory of $\TLBC(q,\Omega)$ with precisely such morphisms.
Note also that $\TLC(q)$ may be regarded  as a subcategory of $\TLBC(q, a_1, a_2)$ in the same way.
\end{remark}

\begin{remark}
If $a_1, a_2$ are as in Lemma \ref{lem:iso-tlcb-2},  then $\TLBC(q, \Omega)\cong\TLBC(q, a_1,  a_2)$.
\end{remark}

\subsubsection{Structure of the Templey-Lieb category of type $B$}\label{sect:TLBC-struct}
It is the category $\TLBC(q, \Omega)$ which is relevant to the study of quantum Schur-Weyl duality in Section \ref{sect:Schur-Weyl}.  
We therefore investigate  its structure in more depth. 

\begin{remark}
All results obtained in this section are also valid for $\TLBC(q, a_1, a_2)$,  the two-parameter Templey-Lieb category of type $B$. 
\end{remark}

Morphisms of $\TLBC(q, \Omega)$ are linear combinations of diagrams with only one vertical thick arc placed 
at the left end, which can be described explicitly as follows.  
\begin{definition}
Call a tangle diagram in $\TLBC(q, \Omega)$ a Temperley-Lieb diagram if it satisfies the following conditions: 
\begin{enumerate}
\item the diagram has only one vertical thick arc placed at the left end, which will be called the pole;
\item there are no loops in the diagram;
\item arcs do not self-tangle, and thin arcs do not tangle with thin arcs; 
\item if a thin arc tangles  with the thick arc, it crosses the thick arc just twice, and crosses behind the pole
in the upper crossing. 
\end{enumerate}
\end{definition}

For example, the diagrams in Figure \ref{fig:seven-five} and Figure \ref{fig:3-5} are morphisms in $\TLBC(q, \Omega)$. 
The one in Figure \ref{fig:3-5} is not in $\TLC(q)$ while that in Figure \ref{fig:seven-five} is in $\TLC(q)$ (see Remark \ref{rem:tlsub}).

\begin{figure}[h]
\begin{picture}(160, 60)(-30,0)

{
\linethickness{1mm}
\put(-15, 0){\line(0, 1){60}}
}

\qbezier(20, 60)(50, 10)(80, 60)
\qbezier(30, 60)(50, 25)(70, 60)

\qbezier(0, 60)(20, 40)(80, 0)

\qbezier(0, 0)(30, 40)(60, 0)

\qbezier(15, 0)(30, 20)(45, 0)

\qbezier(90, 0)(105, 40)(120, 0)
\end{picture}
\caption{A diagram $(m, v^7)\to(m, v^5)$}
\label{fig:seven-five}
\end{figure}

\begin{figure}[h]
\begin{picture}(150, 100)(-50,0)
{
\linethickness{1mm}
\put(-15, 25){\line(0, 1){75}}
\put(-15, 0){\line(0, 1){16}}
}

\qbezier(10, 100)(6, 80)(-11, 75)
\qbezier(-5, 100)(-5, 80)(-11, 80)


\qbezier(-19, 74)(-30, 70)(-30, 50)
\qbezier(-15, 22)(-30, 26)(-30, 50)

\qbezier(-19, 80)(-40, 75)(-38, 40)
\qbezier(-15, 19)(-40, 25)(-38, 50)
\qbezier(-15, 19)(25, 15)(25, 0)
\qbezier(-5, 0)(3, 20)(10, 0)
\qbezier(-15, 22)(40, 10)(50, 100)
\qbezier(20, 100)(30, 20)(40, 100)
\qbezier(35, 0)(50, 45)(60, 100)
\end{picture}
\caption{A diagram $(m, v^4)\to(m, v^6)$}
\label{fig:3-5}
\end{figure}

\noindent 
Figure \ref{fig:seven-five} is a morphism $(m, v^7) \to (m, v^5)$,   where arcs do not cross, and  
Figure \ref{fig:3-5} is a morphism $(m, v^4) \to (m, v^6)$,  where $2$ thin arcs over cross  the thick arc twice each.  
Note that there is a unique Templey-Lieb diagram $m\to m$ consisting of a thick arc only.

The spaces of morphisms of $\TLBC(q, \Omega)$ are easily seen to be spanned by Temperley-Lieb diagrams, since the relations
in Section \ref{sect:TLBC} may be used to reduce any diagram to a linear combination of Temperley-Lieb diagrams.  

Composition of morphisms may be described as follows. Given morphisms 
$D: B\to T$ and $D': T\to U$ represented by \tl diagrams, their composition is defined by the following steps 
\begin{enumerate}
\item Concatenation of diagrams. Concatenate the diagrams $D$ and $D'$ by joining the points on the top of $D$ with those on the bottom of $D'$. 
\item Reduction to Temperley-Lieb diagrams.  Apply locally the skein relation, free loop removal and 
tangled loop removal to turn the resulting diagram into a linear combination of  Temperley-Lieb diagrams $B\to U$.

\item The result of step (2) is the composition $D\circ D'$ of the morphisms $D$ and $D'$. 
\end{enumerate}

The following result is obtained by repeatedly applying the straightening relations.
\begin{lemma}\label{lem:hom-iso}
Let $N$ be any non-negative integer. Then for $r=0, 1, \dots, 2N$, the vector spaces  
$
\Hom_{\TLBC(q, \Omega)}((m, v^r), (m, v^{2N-r}))
$
are all isomorphic. 
\end{lemma}
\begin{proof} The proof is the same as in \cite{LZ14}. 
\end{proof}

Consider in particular $W(2N):=\Hom_{\TLBC(q, \Omega)}(m, (m, v^{2N}))$.  Let $F_tW(2N)$ be the subspace 
of $W(2N)$ spanned by diagrams satisfying that at most $t$ thin arcs are entangled with the pole.  
For example, Figure \ref{fig:0-12} is a diagram $m\to(m, v^{12})$,  
\begin{figure}[h]
\begin{picture}(150, 140)(-50,-40)
{
\linethickness{1mm}
\put(-15, 25){\line(0, 1){75}}
\put(-15, 0){\line(0, 1){16}}
\put(-15, -18){\line(0, 1){18}}
\put(-15, -40){\line(0, 1){16}}
}

\qbezier(10, 100)(6, 80)(-11, 75)
\qbezier(0, 100)(-5, 80)(-11, 80)

\qbezier(-19, 74)(-30, 70)(-30, 50)
\qbezier(-15, 22)(-30, 26)(-30, 50)

\qbezier(-19, 80)(-40, 75)(-38, 40)
\qbezier(-15, 19)(-40, 25)(-38, 50)

\qbezier(-15, 19)(70, 5)(110, 100)

\qbezier(60, 100)(55, 25)(100, 100)

\qbezier(70, 100)(65, 45)(90, 100)
\qbezier(-15, 22)(40, 10)(50, 100)
\qbezier(20, 100)(30, 20)(40, 100)

\qbezier(-15, -22)(-40, -12)(-19, 0)
\qbezier(-11, 1)(100, 3)(120, 100)
\qbezier(-15, -22)(100, -20)(140, 100)
\end{picture}
\caption{A diagram $m\to(m, v^{12})$}
\label{fig:0-12}
\end{figure}

\noindent
where $3$ thin arcs are entangled with the pole.  Note that two of those three (which cross the pole at its upper part)  
arcs are ``parallel'' when they cross the pole.  Let us make this notion more precise.

\begin{definition} Any thin arc which is entangled with the pole has two polar crossings, and in this way defines an interval on the pole,
which we call its ``polar interval''.
We say that two thin arcs in a Temperley-Lieb diagram are parallel if they are both entangled with the pole and
the polar interval defined by one is contained in the polar interval defined by the other. 
\end{definition}

\begin{remark}\label{rem:int}
There is a partial order on the thin arcs in a \tl diagram which are entangled with the pole, which is defined by containment
of their corresponding polar intervals. Two such arcs are parallel if they are comparable in this partial order.
\end{remark}

\begin{definition}\label{def:stand-disting}
Call a Temperley-Lieb diagram $m\to (m, v^{2N})$ {\em distinguished} if any pair of thin arcs which entangle the pole 
are parallel, and {\em standard} if no two thin arcs tangling with the pole are parallel.  In view of Remark \ref{rem:int},
a diagram is distinguished (resp. standard) if those among its arcs which entangle the pole are totally ordered (resp. have no 
two arcs which are comparable) in the partial order on pole-entangled arcs.
\end{definition}

Examples of distinguished and standard diagrams are given respectively in Figure \ref{fig:0-10-d} and in Figure \ref{fig:0-12-s}. 

It is easily seen that there is a unique distinguished (resp. standard) Temperley-Lieb diagram $m\to (m, v^{2N})$ 
with $N$ thin arcs tangling with the pole. 

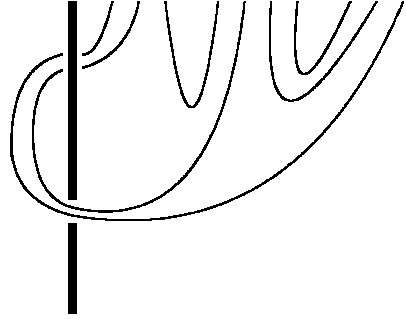
\begin{figure}[h]
\begin{picture}(150, 120)(-50,-20)
{
\linethickness{1mm}
\put(-15, 25){\line(0, 1){75}}
\put(-15, 0){\line(0, 1){16}}
\put(-15, -18){\line(0, 1){18}}
}

\qbezier(10, 100)(6, 80)(-11, 75)
\qbezier(0, 100)(-5, 80)(-11, 80)

\qbezier(-19, 74)(-30, 70)(-30, 50)
\qbezier(-15, 22)(-30, 26)(-30, 50)

\qbezier(-19, 80)(-40, 75)(-38, 40)
\qbezier(-15, 19)(-40, 25)(-38, 50)

\qbezier(-15, 19)(70, 5)(110, 100)

\qbezier(60, 100)(55, 25)(100, 100)

\qbezier(70, 100)(65, 45)(90, 100)
\qbezier(-15, 22)(40, 10)(50, 100)
\qbezier(20, 100)(30, 20)(40, 100)

\end{picture}
\caption{A distinguished diagram $m\to(m, v^{10})$}
\label{fig:0-10-d}
\end{figure}

\begin{figure}[h]
\begin{picture}(150, 140)(-50,-40)
{
\linethickness{1mm}
\put(-15, 65){\line(0, 1){35}}
\put(-15, 25){\line(0, 1){35}}
\put(-15, 0){\line(0, 1){16}}
\put(-15, -18){\line(0, 1){18}}
\put(-15, -40){\line(0, 1){16}}
}

\qbezier(50, 100)(46, 50)(-11, 45)

\qbezier(-5, 100)(-2, 90)(-11, 82)
\qbezier(10, 100)(8, 70)(-15, 63)

\qbezier(-19, 79)(-40, 63)(-15, 63)

\qbezier(-15, 19)(-35, 30)(-19, 42)


\qbezier(-15, 19)(70, 5)(110, 100)

\qbezier(60, 100)(55, 25)(100, 100)

\qbezier(70, 100)(65, 45)(90, 100)
\qbezier(20, 100)(30, 35)(40, 100)

\qbezier(-15, -22)(-40, -12)(-19, 0)
\qbezier(-11, 1)(100, 3)(120, 100)
\qbezier(-15, -22)(100, -20)(140, 100)
\end{picture}
\caption{A standard diagram $m\to(m, v^{12})$}
\label{fig:0-12-s}
\end{figure}

Recall that $F_tW(2N)$ is the subspace of $W(2N)=\Hom_{\TLBC(q, \Omega)}(m, (m, v^{2N}))$ spanned by diagrams with at most $t$
thin arcs which are entangled with the pole.

\begin{lemma} \label{lem:reduct}
\begin{enumerate}
\item Any Temperley-Lieb diagram in $F_tW(2N)$, which does not belong to $F_{t-1}W(2N)$,  can be expressed as a linear 
combination of  elements of $F_{t-1}W(2N)$ and a distinguished (resp. standard) diagram  with $t$ thin arcs tangling with the pole, 
where the coefficient of the distinguished (resp. standard) diagram is a (half-integer) power of $q$. 
\item Any Temperley-Lieb diagram in $F_tW(2N)$ can be expressed as a linear combination of  distinguished (resp. standard) diagrams. 
\end{enumerate}
\end{lemma}
\begin{proof}
By stretching the entangled arcs, one sees that
any Temperley-Lieb diagram diagram  $m\to (m, v^{2N})$  with $t$ thin arcs entangling the pole can be expressed as the 
composition $D'\circ D$ of two diagrams $D:  m\to (m, v^{2t})$  and $D':  (m, v^{2t}) \to (m, v^{2N})$, where  $D'$ has
no thin arcs which entangle the pole.

Now consider a diagram $D:  m\to (m, v^{2t})$ with $t$ thin arcs entangling the pole. If $D$ is distinguished, then the original 
diagram is distinguished. If $D$ is not distinguished, we may pull the thin arcs one by one so that all $t$ arcs are parallel,
at the expense of introducing crossings, which after sliding, will all be to the right of the pole.  
We now use skein relations to remove the crossings among the thin arcs. This leads to a 
linear combination of diagrams, where the only diagram with $t$ thin arcs entangling the pole is the distinguished one, and its coefficient is a power of $q$. 

In a similar way, we can express any Temperley-Lieb diagram in $F_t(2N)$   as a linear combination of a standard diagram and diagrams in $F_{t-1}(2N)$. 
This proves part (1) of the Lemma.

To see (2), observe that
given any Temperley-Lieb diagram diagram  $m\to (m, v^{2N})$  with $t$ thin arcs entangling the pole, we may use 
part (1) to express it as a linear combination of a distinguished (resp. standard) diagram and diagrams in $F_{t-1}(2N)$. 
We then repeat the reduction process for diagrams in $F_{t-1}(2N)$, then $F_{t-2}(2N)$, etc., and after $t$ iterations,
arrive at the statement (2).  
\end{proof}
The $r=N$ case of the following result is stated in \cite{GL98, GL03}. The general case can be deduced from this using Lemma \ref{lem:hom-iso}. 

\begin{lemma} \label{lem:dim-TLB} For all $r=0, 1, \dots, 2N$, 
\[
\dim\Hom_{\TLBC(q, Q)}((m, v^r), (m, v^{2N-r}))=\begin{pmatrix}2N\\ N\end{pmatrix}.
\]
\end{lemma}
\begin{proof}  We give a proof here which will be useful for proving Theorem \ref{thm:main}. 
By Lemma \ref{lem:hom-iso}, we only need to prove the dimension formula for $r=0$. 
Recall that in the proof of Lemma \ref{lem:hom-iso} we introduced the following filtration for $W(2N)=\Hom_{\TLBC(q, Q)}(m, (m, v^{2N}))$.
\begin{eqnarray}\label{eq:filtr}
F_NW(2N)\supset F_{N-1}W(2N)\supset \dots\supset F_1W(2N)\supset F_0W(2N)\supset\emptyset,
\end{eqnarray}
where $F_tW(2N)$ is the subspace spanned by \tl diagrams with at most $t$ thin arcs entangled with the pole.
Let $\End(2N)= \Hom_{\TLBC(q, Q)}((m, v^{2N}), (m, v^{2N}))$ and denote by $\End^0(2N)$ the subspace of 
$\End(2N)$ spanned by diagrams without tanglement.  Then $\End^0(2N)$ is a subalgebra of  
$\End(2N)$ isomorphic to the Temperle-Lieb algebra $\TL_{2N}(q)$ of degree $2N$. Now $\End^0(2N)$  
 acts naturally on $W(2N)$, and the $F_tW(2N)$ are 
$\End^0(2N)$-submodules. Clearly $\frac{F_tW(2N)}{F_{t-1}W(2N)}$ is isomorphic to a cell module 
$W_{2t}(2N)$ \cite[Def.(2.2)] {GL98} for $\End^0(2N))$, which is a simple module in the present generic context. 
 Hence as vector space, 
\begin{eqnarray}
\Hom_{\TLBC(q, Q)}(m, (m, v^{2N})) \cong \bigoplus_{t=0}^N W_{2t}(2N).
\end{eqnarray}
Using  the well-known fact \cite{ILZ1} that 
\[
\dim W_{2t}(2N)=\begin{pmatrix}2N\\ N-t\end{pmatrix} - \begin{pmatrix}2N\\ N-t-1\end{pmatrix}  
\]
with $\dim W_{2N}(2N)=\begin{pmatrix}2N\\ 0\end{pmatrix} =1$, 
we obtain the result. 
\end{proof}

\subsubsection{The Temperley-Lieb algebra of type $B$}

Given any object $A$ in $\TLBC(q, \Omega)$, we let $\End(A):=\Hom_{\TLBC(q, \Omega)}(A, A)$. 
This forms an associative algebra with multiplication given by composition of morphisms.  

\begin{definition} \label{def:algs} For any $r$, let $\TLBC_r(q, \Omega)= \End(m, v^r)$. This is referred to as
a Temperley-Lieb algebra  of type $B$.
\end{definition}

The following lemma justifies the terminology.

\begin{lemma}\label{lem:TLB-alg-iso}
There is an algebra isomorphism 
\[
F: \TLB_r(q, Q)\stackrel{\sim}\longrightarrow \TLBC_r\left(q, \frac{Q}{\sqrt{-1}}\right)
\]
such that, for all $i=1, 2, \dots, r-1$, 
\[
\begin{aligned} 
&\begin{picture}(150, 70)(-20,0)
\put(-85, 30){$y_1\mapsto  q\sqrt{-1}$}
{
\linethickness{1mm}
\put(-15, 0){\line(0, 1){18}}
\put(-15, 22){\line(0, 1){38}}
}
\qbezier(0, 60)(-7, 50)(-12, 42)
\qbezier(-18, 38)(-25, 30)(-15, 20)
\qbezier(-15, 20)(-15, 20)(0, 0)
\put(20, 0){\line(0, 1){60}}
\put(40, 0){\line(0, 1){60}}
\put(60, 30){......}
\put(100, 0){\line(0, 1){60}}
\put(105, 0){, }
\end{picture}\\
&\begin{picture}(150, 70)(-20,0)
\put(-50, 28){$e_i\mapsto$}
{
\linethickness{1mm}
\put(-15, 0){\line(0, 1){60}}
}
\put(0, 0){\line(0, 1){60}}
\put(20, 0){\line(0, 1){60}}
\put(5, 30){...}
\qbezier(40, 0)(50, 50)(60, 0)
\qbezier(40, 60)(50, 10)(60, 60)
\put(80, 0){\line(0, 1){60}}
\put(100, 0){\line(0, 1){60}}
\put(85, 30){...}
\put(36, -10){\small$i$}
\put(52, -10){\small{$i$+1}}
\put(105, 0){. }
\end{picture}
\end{aligned}
\]
\end{lemma}
\begin{proof} From the definitions of $\TLB_r(q, Q)$ and $\TLBC_r\left(q, \frac{Q}{\sqrt{-1}}\right)$, 
it is clear that the map $F$ respects the relations among the $e_i$, and also $F(y_1) F(e_i) = F(e_i) F(y_1)$ for all $i\ge 2$. 
Thus in order to prove the lemma, one only needs to show that $F$ also preserves the relations 
between $y_1$ and $e_1$ in Definition \ref{def:TLB-alg}. 

Note that the third skein relation in the definition of $\ATLC(q, \Omega)$ leads to
\[
F(y_1)^2 = \left(\sqrt{-1}\Omega- (\sqrt{-1}\Omega)^{-1}\right) F(y_1)+ \id, 
\]
where $\id$ is the identity of $\TLBC_r\left(q, \Omega\right)$.  By setting $Q={\sqrt{-1}}\Omega$ in this equation, we obtain 
\begin{align}
&(F(y_1)-Q)(F(y_1)+Q^{-1})=0.  \label{eq:F-skein}
 \end{align}
Also, tangled loop removal in $\ATLC(q, \Omega)$ leads to 
\begin{align}
F(e_1) F(y_1) F(e_1) = -q\sqrt{-1} (\Omega+\Omega^{-1}) F(e_1) = -q(Q-Q^{-1}) F(e_1).
\label{eq:F-inv}
 \end{align}
 Therefore the defining relations in Definition \ref{def:TLB-alg}
are indeed preserved by $F$, completing the proof of the lemma. 
\end{proof}


%
%
\subsection{Interrelationships among various diagram categories} \label{sect:summary}

We summarise the interrelationships among various categories which have arisen in the process of 
developing a new formulation of the \tl category of type $B$ introduced in \cite{GL03}. The relevant categories are:
\begin{itemize}
\item $\TLBB(q,Q)$,  the \tl category of type $B$ introduced in \cite{GL03};
\item $\RTC$, a category of coloured un-oriented tangle diagrams; 
\item $\ATLC(q)$,  the affine \tl category; 
\item $\ATLC(q, a_1, a_2)$,  the multi-polar \tl category with two parameters;
\item $\TLBC(q, a_1, a_2)$,  the \tl category of type $B$ with two parameters;
\item $\ATLC(q, \Omega)$,  the multi-polar \tl category with one parameter;
\item $\TLBC(q, \Omega)$,  the \tl category of type $B$; and
\item $\TLC(q)$,  the \tl category.
\end{itemize}

Their interrelationships are depicted in the commutative diagram below.

\begin{figure}[h]
\begin{tikzpicture}[scale=1.2]
\draw(0, 4) node {$\RTC$};
\draw(3,4) node {$\ATLC(q)$};
\draw(8,4) node {$\TLC(q)$};
\draw(3,2) node {$\ATLC(q, a_1, a_2)$};
\draw(8,2) node {$\ATLC(q,\Omega)$};
\draw(8,0) node {$\TLBC(q,\Omega)$};
\draw(3,0) node {$\TLBC(q, a_1, a_2)$};
\draw(0,-2) node {$\TLC(q)$};
\draw(8, -2) node {$\mathbb{TLB}(q, \sqrt{-1}\Omega)$};
\draw [->] (0.3,4) -- (2.4, 4);
%
\draw [->] (3, 3.7) -- (3, 2.3);
\draw [->] (3.2, 3.7) -- (7.8, 2.3);
\draw [->] (4.1, 2) -- (7.1, 2);
\draw [>->] (3, 0.3) -- (3, 1.7);
\draw [>->] (8, 0.3) -- (8, 1.7);
\draw [->] (4.2, 0) -- (7.1, 0);
\draw[>->](0.4,-1.7)--(3, -0.2);
\draw[>->](0.6,-1.7)--(7.8, -0.2);
\draw[>->](0.2,-1.7)--(2.8, 1.7);
\draw[>->](0,-1.7)--(2.8, 3.8);
\draw[>->](7.4, 4)--(3.6, 4);
\draw[>->](8, 3.7)--(8, 2.3);
\draw[->](8, -1.7)--(8, -0.3);
\draw[>->](0.7, -2)--(6.7, -2);

\draw(5.5, 2.2) node {$\text{specialisation}$};
\draw(5.5, 0.2) node {$\text{specialisation}$};
\draw(7.8, -1) node {$\simeq$};
\end{tikzpicture}
\caption{Relationships among categories}
\label{fig:cats}
\end{figure}
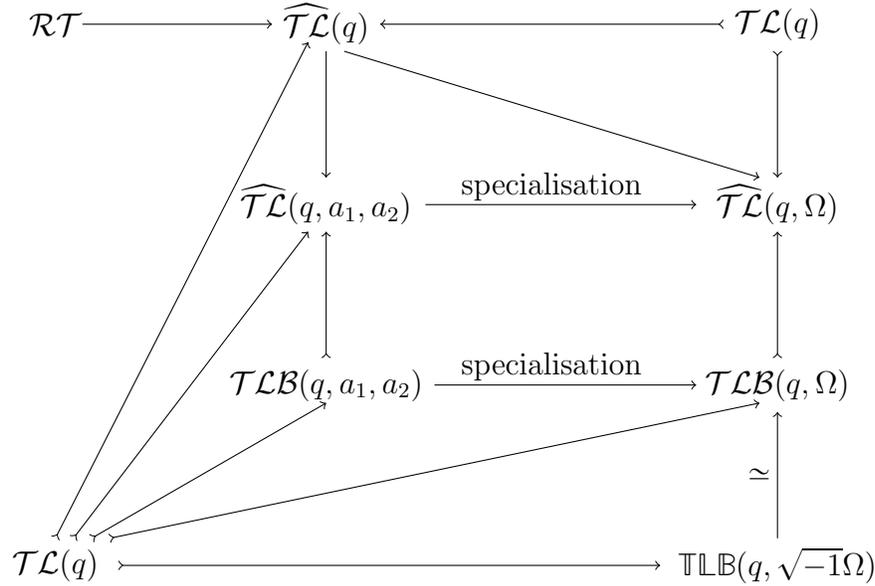
We note that the affine \tl category and the multi-polar \tl categories all contain two different copies of the \tl category $\TLC(q)$ as subcategories. The one at the top left corner of the diagram is always a full subcategory, while the other (at the bottom right corner) is a subcategory as described in Remark \ref{rem:tlsub}.

\subsection{Standard diagrams, tensor product and generators}\label{sss:tp}
The diagrams in the subcategory $\TLC(q)$ of $\TLBC(q,\Omega)$ are those which involve no entanglement  with the pole.
As such, they look like a disentangled pole adjacent to a `usual' finite Temperley-Lieb diagram (cf. \cite{GL98,GL03}).
Now there is an obvious functor 
\be\label{eq:tp2}
\TLBC(q,\Omega)\times\TLC(q)\lr \TLBC(q,\Omega)
\ee
which is defined by juxtaposition of diagrams, where the disentangled pole is omitted from  the second factor.
Moreover it is clear that the diagrams depicted as $I,A,U$ below generate the subcategory 
$\TLC(q)$ under composition and the tensor product defined by \eqref{eq:tp2}.
\begin{figure}[h]
\begin{picture}(130, 90)(-70,-5)
\put(-40, 40){$I=$}
{
\linethickness{1mm}
\put(-5, 0){\line(0, 1){80}}
}
\put(15, 0){\line(0, 1){80}}
\put(20, 0){,}
\end{picture}
\begin{picture}(120, 90)(-35,-5)
\put(-40, 40){$A=$}
{
\linethickness{1mm}
\put(-5, 0){\line(0, 1){80}}
}
\qbezier(15, 0)(25, 150)(35, 0)
\put(40, 0){,}
\end{picture}
\begin{picture}(100, 90)(-35,-5)
\put(-40, 40){$U=$}
{
\linethickness{1mm}
\put(-5, 0){\line(0, 1){80}}
}
\qbezier(15, 80)(25, -70)(35, 80)
\put(40, 0){.}
\end{picture}
\caption{Generators of a $\TLC(q)$ subcategory}
\label{fig:generators-pole}
\end{figure}


We wish to determine similar generators of the whole of $\TLBC(q,\Omega)$. For this we begin by observing that any 
Temperley-Lieb diagram may be expressed as a linear combination of diagrams which are {\em standard} 
in the following sense. 

\begin{definition}\label{def:stand-general}
A Temperley-Lieb diagram $D: (m, v^r)\to (m, v^s)$ is called {\em standard} if 
$(D\otimes \id_{v^r})(\id_m\otimes {\mathbb U}_r): m\to (m, v^{r+s})$ is a standard diagram  in the sense of Definition
\ref{def:stand-disting}, 
where ${\mathbb U}_r: \emptyset\to v^{2r}$ is given by the following diagram. 
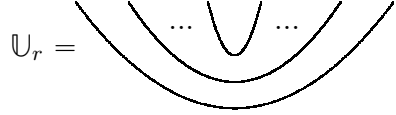
\[
\begin{picture}(120, 70)(-35,30)
\put(-35, 60){${\mathbb U}_r=$}
\qbezier(-10, 80)(50, 0)(110, 80)
\qbezier(10, 80)(50, 20)(90, 80)
\put(25, 70){...}
\qbezier(40, 80)(50, 40)(60, 80)
\put(65, 70){...}
\end{picture}
\]
\end{definition}

It immediately follows from part (2) of Lemma \ref{lem:reduct} that 
\begin{corollary}
Any Temperley-Lieb diagram can be expressed as a linear combination of standard diagrams.
\end{corollary}


It is now clear that to obtain $\TLBC(q,\Omega)$ just one extra generator, depicted as $L$ in the diagram below, needs to be added to the set of generators of $\TLC(q)$ given in Figure \ref{fig:generators-pole}. 
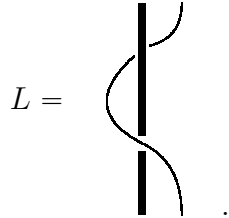
\[
\begin{picture}(150, 90)(-80,-5)
\put(-65, 40){$L=$}
{
\linethickness{1mm}
\put(-15, 0){\line(0, 1){24}}
\put(-15, 30){\line(0, 1){50}}
}
\qbezier(-18, 60)(-40, 40)(-15, 27)
\qbezier(-15, 27)(0, 20)(0, 0)
\qbezier(0, 80)(0, 67)(-12, 64)
\put(15, 0){.}
\end{picture}
\]
We shall use these generators in the proof of Theorem \ref{thm:tlbequ} below.

Note that among the relations in $\TLBC(q,\Omega)$ we have (see ``skein relations'' and ``tangled loop removal'' above)
\be\label{eq:skein}
qL^2=(\Omega+\Omega\inv)L-q\inv I
\ee
and
\be\label{eq:loop}
A(L\ot I)U=-(\Omega+\Omega\inv).
\ee


\section{Quantum Schur-Weyl duality}\label{sect:Schur-Weyl}
\subsection{More on the Quantum group $\U_q(\fsl_2)$}

We give a more detailed account of the representations and the universal $R$-matrix of $\U_q(\fsl_2)$ in this section. 
We write $\Uq$ for the $\CK_0$-algebra $\U_q(\fsl_2)$. This has generators $E,F$ and $K^{\pm 1}$, with relations
\[
KEK\inv=q^2E,\quad  KFK\inv=q^{-2}F, \quad EF-FE=\frac{K-K\inv}{q-q\inv}.\] 

The comultiplication is given by
\[
\Delta(K)=K\ot K,\quad \Delta(E)=E\ot K+1\ot E,\quad \Delta(F)=F\ot 1+K\inv\ot F,
\]
and the antipode by
\[
S(E)=EK\inv,\quad S(F)=-KF,\quad S(K)=K\inv.
\]

\subsubsection{Representations}{\  }\\
{\em Projective modules}.
An integral weight $\U_q(\fsl_2)$-module $M$ of type-${\bf 1}$ is one such that $M=\oplus_{k\in \Z}M_k$ where 
$M_k=\{v\in M\ K v= q^k v \}$. 
Let $\U_q(\fb)$ be the Borel subalgebra of $\U_q(\fsl_2)$ generated by $E$ and $K^{\pm 1}$.  
The category $\CO_{int}$ is defined as the category of  $\U_q(\fsl_2)$-modules $M$, which satisfy:
\begin{itemize}
\item $M$ is finitely generated as a $\U_q(\fsl_2)$-module.
\item $M$ is locally $\U_q(\fb)$ finite.
\item $M$ is an integral weight module of type ${\bf 1}$. 
\end{itemize}

For any integer $\ell\in\Z$, denote by $(\CK_0)_\ell = \CK_0v^+$ the $1$-dimensional $\U_q(\fb)$-module such that 
$E v^+=0$ and $K v^+ = q^{\ell} v^+$, and let $M(\ell)=\U_q(\fsl_2)\otimes_{\U_q(\fb)}(\CK_0)_\ell$. This is the Verma module with highest weight $\ell$, which has a unique 
simple quotient $V(\ell)$. Then $M(\ell)$ and $V(\ell)$ are the standard and simple objects in $\CO_{int}$. 
The simple module $V(\ell)$ is finite dimensional if and only if $\ell\ge 0$, and in this case it is $(\ell+1)$-dimensional. 

The quantum group $\U_q(\fsl_2)$ has a central element $z:=FE+\frac{qK+q^{-1}K^{-1}}{(q-q^{-1})^2}$. 
It acts as the scalar $\chi(\ell):=\frac{q^{\ell+1}+q^{-\ell-1}}{(q-q^{-1})^2}$ on a highest weight vector of weight $\ell$, and since 
it is central, therefore acts as the scalar $\chi(\ell)$ on the whole of any highest weight module in $\CO_{int}$ with highest weight $\ell$. 

For any module $M\in\CO_{int}$, if we 
define $M^{\chi_\ell}:=\{m\in M\mid (z-\chi_\ell)^im=0\text{ for some }i\geq 0\}$,
then $M^{\chi_\ell}$ is a direct summand of $M$. We denote by $\CO_{int}^{\chi_\ell}$ the full subcategory of $\CO_{int}$ whose objects are modules $M$ such that $M=M^{\chi_\ell}$, and call it the block (which is indeed a block)
of $\CO_{int}$ corresponding to $\chi_\ell$.
It follows from the definition of $\CO_{int}$ that any $M\in\CO_{int}$ is an object in a direct sum of finitely many blocks.

Clearly $\chi(\ell) = \chi(\ell')$ if and only if  $\ell=\ell'$  or $\ell+\ell'+2=0$.  

\begin{definition}\label{def:link}
The weights $\ell,\ell'\in\Z$ are said to be linked if $\ell=\ell'$ or $\ell+\ell'+2=0$. The linkage principal asserts here merely
that $\chi(\ell)=\chi(\ell')\iff$ $\ell,\ell'$ are linked.
\end{definition}

Observe that if $\ell\ge -1$, then there exists no $\ell'>\ell$ linked to $\ell$. 
This leads to the following result, which is well known, but we provide a proof for the convenience of the reader.

\begin{lemma} \label{lem:proj} Fix an integer $\ell\ge -1$.
\begin{enumerate}
\item The Verma module $M(\ell)$ is projective in $\CO_{int}$. 
\item If $M$ is a finite dimensional module in $\CO_{int}$, then
\[
\Hom_{\U_q(\fsl_2)}(M(\ell), M(\ell)\otimes M) \cong M_0, 
\]
where $M_0$ is the zero weight space of $M$. 
\end{enumerate}
\end{lemma}

\begin{proof}
Given the linkage principle in Definition \ref{def:link}, the usual arguments from the context of Lie algebras may be adapted to prove the lemma. 

We first consider part (1). 
Let $\psi: M\twoheadrightarrow   N$ be any surjection in $\CO_{int}$. Then $\psi(M)= N$, so that $\psi(M^{\chi_\ell})=N^{\chi_\ell}$.  
If $\phi: M(\ell)\longrightarrow N$, then clearly $\phi(M(\ell))\subseteq N^{\chi_\ell}$ and hence to prove (1), we may suppose that $M$ and $N$ are
in the block $\CO_{int}^{\chi_\ell}$ of $\CO_{int}$. 
Since the image $\phi(m_+)$ of the highest weight vector $m_+$ of $M(\ell)$ is in $N=\psi(M)$, 
 $\phi(m_+)=\psi(v)$ for some
$v\in M$. Writing $v=v_0+v_1+\dots+v_k$, where $v_0$ has weight $\ell$ and for $i\geq 1$, $v_i$ has weight $\ell_i\neq \ell$
(the $\ell_i$ being pairwise distinct), we see that for $j=1,2,3,\dots, $ we have 
$q^{j\ell}(\psi(v)-\psi(v_0))+q^{j\ell_1}\psi(v_1)+\dots + q^{j\ell_k}\psi(v_k)=0$. A van der Monde type argument shows that
$\psi(v_i)=0$ for $i>0$ and that $\psi(v)=\psi(v_0)$. Replacing $v$ by $v_0$, we may therefore assume that $v$ is a weight vector 
of weight $\ell$.

Now the subspace $\U_q(\fb) v\subseteq M$ contains a highest weight vector $v'$ of weight (say) $\ell'$. Thus 
$v'$ is an eigenvector of $z$, with eigenvalue $\chi_{\ell'}$. Since $M$
is in the block $\CO_{int}^{\chi_\ell}$, we must have $\chi_\ell=\chi_{\ell'}$, i.e. $\ell'$ is linked to $\ell$. 
But $\ell'\ge \ell\ge -1$,  so $\ell'=\ell$, and hence $v$ is a highest weight vector in $M$. 

It follows that the unique
homomorphism $\phi': M(\ell)\longrightarrow M$ with $\phi'(m_+)=v$,  renders the following diagram commutative. 
\[
\xymatrix{ 
&M(\ell)\ar[d]^\phi\ar@{-->}[ld]_{\phi'}\\
M\ar@{->>}[r]_\psi &N.
}
\]
This proves that $M(\ell)$ is projective in $\CO_{int}$. 

We now prove part (2). 
Since the module $M$ given in part (2) is finite dimensional, $M(\ell)\otimes M= \U_q(\fsl_2)\otimes_{\U_q(\fb)}((\CK_0)_\ell\otimes M)$. 
Applying the induction functor to the composition series of $(\CK_0)_\ell\otimes M$ as $\U_q(\fb)$-module leads to a filtration 
\[
M(\ell)\otimes M:=W_0\supset W_1\supset W_2 \supset \dots \supset W_{D-1}\supset W_D=0, \quad D=\dim M, 
\]
where the $W_i$ are in $\CO_{int}$ and $W_i/W_{i+1}=M(\ell_i)$ for some integer $\ell_i$. By the projectivity of $M(\ell)$, we have 
\[
\Hom_{\U_q(\fsl_2)}(M(\ell), W_i) =\Hom_{\U_q(\fsl_2)}(M(\ell), M(\ell_i))\oplus \Hom_{\U_q(\fsl_2)}(M(\ell), W_{i+1})
\]
 as vector space,  and hence 
\[
\Hom_{\U_q(\fsl_2)}(M(\ell), M(\ell)\otimes M) = \bigoplus_i \Hom_{\U_q(\fsl_2)}(M(\ell), M(\ell_i)). 
\]
But $\Hom_{\U_q(\fsl_2)}(M(\ell), M(\ell_i))\ne 0$ only when $\ell$ and $\ell_i$ are linked. 
The condition $\ell\ge -1$ requires $\ell_i=\ell$, and in this case,  $\Hom_{\U_q(\fsl_2)}(M(\ell), M(\ell_i))$
 is one dimensional.  Since the weight space of weight  $\ell$ in $M(\ell)\ot M$ is $v_+\ot M_0$, we have
$
\Hom_{\U_q(\fsl_2)}(M(\ell), M(\ell)\otimes M) \cong M_0. 
$
\end{proof}

\noindent{\em Some properties of $V=V(1)$}. There exists a basis $\{v_1, v_{-1}\}$ of $V$ such that the corresponding representation is given by 
\[
K \mapsto \begin{pmatrix} q & 0\\ 0 & q^{-1}\end{pmatrix}, \quad E\mapsto \begin{pmatrix} 0 & 1\\ 0 & 0 \end{pmatrix}, 
\quad F\mapsto \begin{pmatrix} 0 & 0\\ 1& 0 \end{pmatrix}. 
\]
We have $V\otimes V=V(2)\oplus V(0)$, with the $1$-dimensional submodule spanned by
\[
c_0:= -q v_1\otimes v_{-1} + v_{-1}\otimes v_1. 
\]
Since $V$ is self-dual, there exists a unique (up to scalar multiple)  non-degenerate  invariant bilinear form 
$(\ , \ ): V\times V\longrightarrow \CK_0$ given by 
\[
(v_1, v_{-1})=-q (v_{-1}, v_1)=1, \quad (v_1, v_1)=(v_{-1}, v_{-1})=0.
\]
Here invariance of the form means that  
$(X v, w) = (v, S(X) w)$ for all $v, w\in V$ and $X\in \U_q(\fsl_2)$, where $S$ is the antipode of $\U_q(\fsl_2)$.

The following maps are clearly $\U_q(\fsl_2)$-morphisms
\begin{eqnarray}\label{eq:curls}
&\check{C}: \CK_0\longrightarrow V\otimes V, \quad 1\mapsto c_0,  \label{eq:cup}\\
&\hat{C}:  V\otimes V \longrightarrow \CK_0, \quad v\otimes w\mapsto (v, w), \label{eq:cap}
\end{eqnarray}
as are the maps $\eta, \zeta: V\longrightarrow V$ respectively defined by the compositions 
\[
\begin{aligned}
\eta: V\stackrel{\sim}\longrightarrow \CK_0\otimes V \stackrel{\check{C}\otimes\id}\longrightarrow V
 \otimes V \otimes V \stackrel{\id\otimes \hat{C}}\longrightarrow V, 
\\
\zeta: V\stackrel{\sim}\longrightarrow V\otimes\CK_0 \stackrel{\id\otimes \check{C}}\longrightarrow V 
\otimes V \otimes V \stackrel{\hat{C}\otimes \id}\longrightarrow V. 
\end{aligned}
\]
The statements in the lemma below are all either well-known  or easily checked.
\begin{lemma} \label{eq:cup-cap} Let $e= \check{C}\circ\hat{C}$. The following relations hold. 
\begin{eqnarray}
 \hat{C}(\check{C})=-(q+q^{-1}), \label{eq:loop-c}\\
\eta=\zeta=\id_V,  \label{eq:straighten}\\
e^2=-(q+q^{-1}) e.  \label{eq:e-op}
\end{eqnarray}
\end{lemma}

\subsubsection{The universal $R$-matrix}
As the universal $R$-matrix of $\U_q(\fsl_2)$ will play an important role in our development, we give some explicit
information concerning it.  Following  \cite{LZ06}, we 
 define a functorial linear operator $\Xi$ as follows.  For any pair of modules $M_1, M_2$ in $\CO_{int}$,  and weight vectors 
$w_1\in M_1$ and $w_2\in M_2$ with weights $k_1, k_2$ respectively,
\begin{eqnarray}
\Xi_{M_1, M_2}: M_1\otimes M_2 \longrightarrow M_1\otimes M_2, \quad w_1\otimes w_2 \mapsto q^{\frac{k_1 k_2}{2}} w_1\otimes w_2.
\end{eqnarray}
The universal $R$-matrix is the functorial linear isomorphism
\begin{eqnarray}\label{eq:univ-R}
R = \Xi \left(\sum_{j=0}^\infty \frac{(q-q^{-1})^j}{\qint{j}!} E^j\otimes F^j\right),
\end{eqnarray}
where $\qint{j}!=\prod_{k=0}^j\qint{k}$ with $\qint{k}=\frac{1-q^{-2k}}{1-q^{-2}}$. [Warning: this is not the usual definition of $q$-numbers.]
For $i=1,2$, denote the representation of $\U_q(\fsl_2)$ on $M_i$ by $\pi_i$. 
Then the universal $R$-matrix acts on $M_1\otimes M_2$ by 
\[
R_{M_1, M_2} = \Xi \left(\sum_{j=0}^\infty \frac{(q-q^{-1})^j}{\qint{j}!} \pi_1(E^j)\otimes \pi_2(F^j)\right).
\]
This is well defined, since $E$ and $F$ act locally nilpotently.  The universal $R$-matrix has the following properties. 
\begin{eqnarray}
&R_{M_1, M_2}  (\pi_1\otimes \pi_2)\Delta(x) =  (\pi_1\otimes \pi_2)\Delta'(x)R_{M_1, M_2}, \quad \forall x\in \U_q(\fsl_2);  \\
&R_{M_1\otimes M_2, M_3} = R_{M_1, M_3} R_{M_2, M_3}, \quad 
R_{M_1, M_2\otimes M_3} = R_{M_1, M_3} R_{M_1, M_2}, \\
&R_{M_1, M_2} R_{M_1, M_3} R_{M_2, M_3} =R_{M_2, M_3}R_{M_1, M_3} R_{M_1, M_2}, 
\end{eqnarray}
where the last two equations are equalities of automorphisms of $M_1\otimes M_2\otimes M_3$. 
The last equation is the celebrated Yang-Baxter equation. 

Let $
P_{M_1, M_2}: M_1\otimes M_2 \longrightarrow M_2\otimes M_1
$
be the permutation $w\otimes w' \mapsto w'\otimes w$,  and denote
\[
\begin{aligned}
&\check{R}_{M_1, M_2}= P_{M_1, M_2}R_{M_1, M_2}: M_1\otimes M_2 \longrightarrow M_2\otimes M_1. \end{aligned}
\]
Then 
\begin{eqnarray}\label{eq:commute}
\check{R}_{M_1, M_2}  (\pi_1\otimes \pi_2)\Delta(x) - (\pi_2\otimes \pi_1)\Delta(x)\check{R}_{M_1, M_2}=0, \quad \forall x\in \U_q(\fsl_2), 
\end{eqnarray}
and the Yang-Baxter equation becomes the  following ``braid relation''  among isomorphisms
$M_1\otimes M_2\otimes M_3\longrightarrow M_3\otimes M_2\otimes M_1$ in $\CO_{int}$.
\[
\begin{aligned}
(\check{R}_{M_2, M_3}\otimes\id_{M_1})  (\id_{M_2}\otimes\check{R}_{M_1, M_3})  
(\check{R}_{M_1, M_2}\otimes\id_{M_3})\\
=
(\id_{M_3}\otimes\check{R}_{M_1, M_2})  (\check{R}_{M_1, M_3}\otimes\id_{M_2})(\id_{M_1}\otimes\check{R}_{M_2, M_3}).   
\end{aligned}
\]

For $M_1=M_2=V(1)$, by looking at the action of  $q^{\frac{1}{2}}\check{R}_{V, V}$ on the respective highest weight 
vectors of the simple submodules of $V\otimes V=V(2)\oplus V(0)$, it becomes evident that $q^{\frac{1}{2}}\check{R}_{V, V}$ 
 has eigenvalues $q$ and $-q^{-1}$ on $V(2)$ and $V(0)$ respectively. Bearing in mind that $\check{C}\circ\hat{C}=-(q+q\inv)$
times the projection to $V(0)$ this may be restated as follows.
\begin{lemma}\label{eq:normal-R} The $R$-matrix $\check{R}_{V, V}$ satisfies the following relation. 
\begin{eqnarray}\label{lem:normal-R}
q^{\frac{1}{2}}\check{R}_{V, V}=q+\check{C}\circ\hat{C},  
\end{eqnarray}
where $\check{C}$ and $\hat{C},$ are as defined in \eqref{eq:curls}. 
\end{lemma}

Now let 
\begin{eqnarray}
R^T = \Xi \left(\sum_{j=0}^\infty \frac{(q-q^{-1})^j}{\qint{j}!} F^j\otimes E^j\right),  \label{eq:univ-RT}
\end{eqnarray}
and denote $R^T_{M_1, M_2}= R^T:  M_1\otimes M_2 \longrightarrow M_1\otimes M_2$. Then 
\[
\begin{aligned}
R^T_{M_1, M_2}= P_{M_2, M_1}R_{M_2, M_1} P_{M_1, M_2}: M_1\otimes M_2 \longrightarrow M_1\otimes M_2.
\end{aligned}
\]
Furthermore, $\check{R}_{M_2, M_1}\check{R}_{M_1, M_2}=R^T_{M_1, M_2}R_{M_1, M_2}$, and 
\[
\begin{aligned}
R^T_{M_1, M_2}R_{M_1, M_2}
&= \Delta(v^{-1}) (v\otimes v): M_1\otimes M_2 \longrightarrow M_1\otimes M_2,
\end{aligned}
\]
where $v$ is Drinfeld's central element of $\U_q(\fsl_2)$ (see \cite{LZ06}).  The element $v$ acts on any highest 
weight module with highest weight $\ell$ (i.e., cyclically generated by a highest weight vector of
weight $\ell$) in $\CO_{int}$ as multiplication by the scalar  $q^{ -\frac{1}{2}\ell (\ell+2)}$.

Assume that $M_1$ and $M_2$ are both highest weight modules with highest weights 
$\ell_1$ and $\ell_2$ respectively. Then $R^T_{M_1, M_2}R_{M_1, M_2}$ acts on a highest weight submodule 
$M'$ of $M_1\otimes M_2$  with highest weight $\ell$ as  
\begin{eqnarray}\label{eq:eigen}
\begin{aligned}
&R^T_{M_1, M_2}R_{M_1, M_2}|_{M'} =  q^{\chi(\ell, \ell_1, \ell_2)} \id_{M'}, \quad \text{where}\\
&\chi(\ell, \ell_1, \ell_2) = \frac{\ell(\ell+2)}{2} -\frac{\ell_1(\ell_1+2)}{2} - \frac{\ell_2(\ell_2+2)}{2}.
\end{aligned}
\end{eqnarray}
If $m_+$ is the highest weight vector of $M_1$, and $v\in M_2$ is a vector of weight $j$, we have
$Km_+=q^{\ell_1}m_+$ and $K v= q^j v$.  Then 
\begin{eqnarray}\label{eq:RtR}
R^T_{M_1, M_2}R_{M_1, M_2}(m_+\otimes v) =\sum_{k=0}^\infty\frac{(q-q^{-1})^k q^{ j \ell_1 
+k(\ell_1-j-2k)}}{\qint{k}!}F^k m_+\otimes E^k v.
\end{eqnarray}

\subsection{A tensor functor}

We again fix $V=V(1)$, and for any integer $\ell$, let $M$ be either $M(\ell)$ or $V(\ell)$. 
We shall adopt the following notation. Recall that $\cC=\{m, v\}$; for any sequence $A=(a_1, a_2, \dots, a_r)$ 
with $a_j\in\cC$, write
\[
U^A = U^{a_1}\otimes U^{a_2}\otimes \dots \otimes U^{a_r},
\]
where $U^m=M$ and $U^v=V$.
We set $U^\emptyset=\CK_0$ for the empty sequence $\emptyset$.

Recall that there exists a canonical tensor functor from the category of directed coloured ribbon 
graphs to the category of finite dimensional representations of any quantum group, see  \cite[Theorem 5.1]{RT}.  
Adapting that functor to our context yields the following result. 

\begin{theorem} \label{thm:RT} Let $V=V(1)$, and for any integer $\ell$ let $M$ be either $M(\ell)$ or $V(\ell)$.
There exists a covariant linear functor   
$
\widehat\CF:  \RTC\longrightarrow \CO_{int},
$
which is defined by the following properties.  
\begin{enumerate}
\item The functor respects the tensor products of $\RTC$ and $\CO_{int}$;  
\item $\widehat\CF$ sends 
the object $A$ of  $\RTC$ to $\widehat\CF(A)=U^A$; and 

\item
$\widehat\CF$ maps the generators of the morphism spaces as indicated below.

\[
\begin{aligned}
\begin{picture}(60, 70)(0,0)
\put(0, 0){\line(0, 1){60}}
\put(-10, 40){$a$}
\put(15, 30){$\mapsto \  \id_{U^a}$, }
\end{picture}
&\quad\quad&
\begin{picture}(80, 70)(-30,0)
\qbezier(-35, 60)(-35, 60)(-15, 0)
\qbezier(-15, 60)(-15, 60)(-22, 33)
\qbezier(-35, 0)(-35, 0)(-26, 25)
\put(-15, 40){$b$ }
\put(-40, 40){$c$ }
\put(0, 30){$\mapsto \  \check{R}_{U^b, U^c}$, }
\end{picture}
&\qquad\qquad&
\begin{picture}(80, 70)(-70,0)
\qbezier(-90, 0)(-90, 0)(-70, 60)
\qbezier(-90, 60)(-90, 60)(-83, 33)
\qbezier(-70, 0)(-70, 0)(-78, 26)
\put(-70, 40){$b$ }
\put(-95, 40){$c$ }
\put(-55, 30){$\mapsto \  (\check{R}_{U^b, U^c})^{-1}$, }
\end{picture}\\
\end{aligned}
\]
for all $a, b,  c \in \cC$ with $b,  c$ not both equal to $m$,  and
\[
\begin{aligned}
\begin{picture}(100, 50)(0, 10)
\qbezier(0, 60)(15, -10)(30, 60)
\put(13, 12){$v$}
\put(45, 35){$\mapsto \ \check{C}$,}
\end{picture}\quad  
\begin{picture}(130, 60)(-10, -10)
\qbezier(0, 0)(15, 70)(30, 0)
\put(12, 40){$v$}
\put(45, 15){$\mapsto \hat{C}$,}
\end{picture}
\end{aligned}
\]
where the maps $\check{C}: \CK_0\longrightarrow V\otimes V$ and $\hat{C}:   V\otimes V \longrightarrow \CK_0$ 
are defined by \eqref{eq:cup} and \eqref{eq:cap} respectively. 
\end{enumerate}
\end{theorem}
\begin{proof} 
Let us prove part (1) of the theorem first. 
The functor clearly respects the tensor products for objects, since for any objects $A$ and $B$ in $ \RTC$, we have 
\[
\widehat\CF(A\otimes B)= U^{(A, B)}= U^A\otimes U^B.
\]
Now we require that $\widehat\CF$ respect tensor products for morphisms, that is, for any two morphisms $D, D'$ in $ \RTC$, 
\be\label{eq:rtp}
\widehat\CF(D\otimes D')=\widehat\CF(D)\otimes \widehat\CF(D').
\ee
Equation \eqref{eq:rtp} and property (3) of the statement, which defines the images of the generators of morphisms under 
$\widehat\CF$,  define the functor uniquely. Thus what remains to be shown is that the functor is well 
defined on morphisms, that is, we need to show that  the relations in part (3) of Theorem \ref{thm:tensor-cat} 
are preserved by $\widehat\CF$. 

Clearly $\widehat\CF$ preserves relations (a). It also preserves relations (b) since the $R$-matrices 
satisfy the Yang-Baxter equation \eqref{eq:YBE}.  It follows from \eqref{eq:straighten} 
that the straightening relations (c) are also preserved. 

To prove the sliding relations, let us denote 
\[
\Psi_+:=\widehat{\CF}\left(
\begin{picture}(80, 30)(-12,15)
\qbezier(0, 45)(0, 45)(22, 0)
\qbezier(0, 0)(10, 13)(12, 15)
\qbezier(17, 21)(35, 50)(60, 0)
\put(-8, 35){$a$}
\put(32, 18){$v$}
\end{picture}
\right),
\qquad \quad
\Psi_+':=\widehat{\CF}\left(
\begin{picture}(80, 30)(-12,15)
\qbezier(0, 0)(25, 50)(45, 20)
\qbezier(49, 15)(50, 15)(60, 0)
\qbezier(60, 45)(60, 45)(38, 0)
\put(62, 35){$a$}
\put(22, 18){$v$}
\end{picture}
\right), 
\]
\[
\Psi_-:=\widehat{\CF}\left(
\begin{picture}(80, 30)(-12,15)
\qbezier(0, 0)(30, 60)(60, 0)
\qbezier(0, 45)(0, 45)(10, 22)
\qbezier(15, 15)(15, 15)(22, 0)
\put(-8, 35){$a$}
\put(32, 18){$v$}
\end{picture}\right), 
\qquad \quad
\Psi'_-:=\widehat{\CF}\left(
\begin{picture}(80, 30)(-8,15)
\qbezier(0, 0)(30, 60)(60, 0)
\qbezier(60, 45)(60, 45)(50, 22)
\qbezier(45, 15)(45, 15)(38, 0)
\put(62, 35){$a$}
\put(22, 18){$v$}
\end{picture}\right).
\]
Let $v, v'\in V$ be vectors with weights $\ell$ and $\ell'$ respectively, and let $w\in U^a$ be a vector with weight $k$.  Then we have 
\[
\begin{aligned}
\Psi_+(v\otimes w\otimes v') &= q^{\frac{1}{2} k \ell} (v, v') w + (q-q^{-1}) q^{\frac{1}{2} (k-2)(\ell+2)}(E v, v') F w, \\
\Psi'_+(v\otimes w\otimes v') &= q^{-\frac{1}{2} k \ell'} (v, v') w - (q-q^{-1}) q^{-\frac{1}{2} k\ell'}(v, E v') F w.
\end{aligned}
\]
The invariance of the bilinear form implies that   $(v, E v')= -(E v, K v')  = -q^{\ell'}(E v, v')$.  Also note that 
$(v, v')=0$ unless $\ell+\ell'=0$,  and  $(E v, v')= (v, E v')=0$ unless $\ell=\ell'=-1$. Hence  
\[
\Psi_+(v\otimes w\otimes v') = q^{\frac{1}{2} k \ell} (v, v') w + (q-q^{-1}) q^{\frac{1}{2} (k-2)}(E v, v') F w=
\Psi'_+(v\otimes w\otimes v'). 
\]
Similarly we have 
\[
\begin{aligned}
\Psi_-(v\otimes w\otimes v') &
=q^{-\frac{1}{2}k \ell}(v, v') w-(q-q^{-1})q^{-\frac{1}{2}k\ell}(Fv,v')Ew,\\
\Psi'_-(v\otimes w\otimes v') &
=q^{\frac{1}{2}k \ell'}(v, v') w+(q-q^{-1})q^{\frac{1}{2}(k +2)(\ell' -2)}(v,Fv')Ew.
\end{aligned}
\]
In this case,  $(Fv,v')=-(v,KFv')=-q^{-1}(v,Fv')$, and  $(Fv,v')=(v,Fv')=0$ unless  $\ell=\ell'=1$.
We still have $(v, v')=0$ unless $\ell+\ell'=0$.  Hence  
\[
\Psi_-(v\otimes w\otimes v') = q^{-\frac{1}{2} k \ell} (v, v') w + (q-q^{-1}) q^{-\frac{1}{2} (k+2)}(v, F v') E w=
\Psi'_-(v\otimes w\otimes v'). 
\]

Now consider the twists.  Let
 \[
 \varphi:=\widehat{\CF}\left(\begin{picture}(50, 40)(-10,80)
 \qbezier(0, 100)(-5, 110)(-5, 120)
\qbezier(0, 100)(0, 100)(20, 70)
\qbezier(0, 70)(0, 70)(8, 82)
\qbezier(0, 70)(-5, 60)(-5, 50)
\qbezier(20, 100)(20, 100)(13, 88)
\qbezier(20, 100)(30, 110)(32, 85)
\qbezier(20, 70)(30, 60)(32, 85)
\put(22, 108){$v$}
\end{picture}
\right),
\qquad\quad 
\varphi':=\widehat{\CF}\left(\begin{picture}(50, 40)(-12,30)
\qbezier(20, 50)(0, 20)(0, 20)
\qbezier(0, 50)(0, 50)(8, 38)
\qbezier(12, 33)(20, 20)(20, 20)
\qbezier(20, 50)(30, 60)(32, 35)
\qbezier(20, 20)(30, 10)(32, 35)
\qbezier(-5, 0)(-5, 10)(0, 20)
\qbezier(0, 50)(-5, 60)(-5, 70)
\put(23, 57){$v$}
\end{picture}\right), 
 \]
which are scalar multiples of $\id_V$.  By direct computations, one can verify that 
\[
 \varphi v_1 = - q^{\frac{3}{2}} v_1, \quad \varphi'v_{-1}=-q^{-\frac{3}{2}} v_{-1}. 
 \]
 Hence $\varphi= - q^{\frac{3}{2}}\id_V$, $\varphi'=-q^{-\frac{3}{2}} \id_V$, and $\varphi\circ\varphi'=\id_V$. 
 Another way to compute this is to use the well known relationship between the $R$-matrix and  Drinfeld's central element. 
By taking into account the $q$-skew nature of the bilinear form,  we immediately obtain 
\[
\varphi= - q^{\frac{3}{2}} \id_V, \quad  \varphi'= - q^{-\frac{3}{2}} \id_V.
\]

This completes the proof of Theorem \ref{thm:RT}.
\end{proof}

\begin{remark} \label{rem:norm-factors} 
Note that one has the freedom of multiplying the $R$-matrices by an invertible 
scalar in the definition of the representation of $\CG_r$ in Lemma \ref{lem:braid-rep}.  
However, this is disallowed by the sliding relations in the definition of  the functor $\widehat{\CF}$.
\end{remark}

\begin{theorem}\label{thm:functor-quot}
The functor $\widehat\CF:  \RTC\longrightarrow \CO_{int}$ of Theorem \ref{thm:RT} 
factors through $\ATLC(q, q^{\ell+1})$.  
\end{theorem}
\begin{proof}
It follows from the property of $\check{R}$ given in Lemma \ref{eq:normal-R} that the functor 
$\widehat\CF$ respects the first two skein relations of $\ATLC(q, \Omega)$.  
It also respects the free loop removal relation by \eqref{eq:loop-c}.  
We now set  
\begin{eqnarray}\label{eq:omega-value}
\Omega=q^{\ell+1}. 
\end{eqnarray}

We want to show that $\widehat\CF$ preserves the third skein relation of $\ATLC(q, \Omega)$. 
We have 
\[
\xi:=\widehat\CF\left(
\begin{picture}(40, 50)(-30,50)
{
\linethickness{1mm}
\put(-15, 65){\line(0, 1){35}}
\put(-15, 25){\line(0, 1){35}}
\put(-15, 0){\line(0, 1){18}}
}
\qbezier(-19, 38)(-35, 30)(-15, 22)
\qbezier(-15, 62)(5, 55)(-11, 42)
\qbezier(-15, 62)(-35, 75)(-19, 85)
\qbezier(0, 100)(-6, 90)(-11, 87)
\qbezier(-15, 22)(-5, 15)(0, 0)
\end{picture}
\right) = \check{R}_{V, M}\check{R}_{M, V}. 
\]
\vspace{.5mm}

In the case $M=M(\ell)$, if $\ell\ne -1$, we have $M\otimes V= M(\ell+1)\oplus M(\ell-1)$.  Hence  
$\xi=\check{R}_{V, M}\check{R}_{M, V}$ has two eigenvalues $q^{\chi_\pm}$, 
which we can compute by using \eqref{eq:eigen} to obtain 
\[
\chi_\pm = \frac{1}{2}(\ell \pm1)(\ell+2\pm 1) - \frac{1}{2}\ell (\ell+2) - \frac{3}{2} =\pm (\ell +1)-1.
\]
Hence $\xi$ satisfies the quadratic relation
\begin{eqnarray}\label{eq:skein-3}
(q\xi-\Omega)(q\xi-\Omega^{-1})=0. 
\end{eqnarray}

If $\ell=-1$, we do not have such a decomposition for $M\otimes V$.  However, we can directly verify the 
skein relation. Since $\check{R}_{V, M}\check{R}_{M, V}$ is a $\U_q(\fsl_2)$-morphism, we only need
 to verify this for the two vectors $m_+\otimes v_1$ and $m_+\otimes v_{-1}$, as they generate $M\otimes V$. 
It is clear that $m_+\otimes v_1$ is an eigenvector of $R^T_{M, V}R_{M, V}$ with eigenvalue $q^{-1}$. 
For the vector $m_+\otimes v_{-1}$, we use \eqref{eq:univ-R} and \eqref{eq:univ-RT} to obtain the following relations. 
\[
\begin{aligned}
R^T_{M, V}R_{M, V}(m_+\otimes v_{-1}) &= q m_+\otimes v_{-1} + q^{-1}(q-q^{-1}) F m_+\otimes v_1,\\
(R^T_{M, V}R_{M, V})^2 (m_+\otimes v_{-1})&= (2-q^{-2}) m_+\otimes v_{-1} + 2q^{-2}(q-q^{-1}) F m_+\otimes v_1.
\end{aligned}
\]
Combining these we arrive at $(\check{R}_{V, M}\check{R}_{M, V}-q^{-1})^2(m_+\otimes v_{-1}) =0$.  
Hence we have proved that   in this case
\[
(q\xi-1)^2=0.
\]

For $M=V(\ell)$, we only need to consider $\ell\ge 0$, since $V(\ell)=M(\ell)$ if $\ell<0$.   
We have $V(\ell)\otimes V= V(\ell+1)\oplus V(\ell-1)$ for $\ell>0$ , and $V(0)\otimes V=V$.  
It follows from these decompositions that the relation \eqref{eq:skein-3} is also satisfied in this case. 

This proves that $\widehat\CF$ preserves the third skein relation for $\ATLC(q, q^{\ell+1})$.

We now verify the tangled loop removal relation for both $M(\ell)$ and $V(\ell)$.  We 
note that for any linear transformation $\phi$ of $V$, 
\[
\hat{C}(\phi\otimes\id_V)\check{C}(1)= -q(v_1 v_{-1}) + (\phi v_{-1}, v_1) = - tr_V(K\phi).
\]
Hence 
\[
\Phi:= \widehat\CF\left(\begin{picture}(75, 50)(-40,45)
{
\linethickness{1mm}
\put(-15, 65){\line(0, 1){35}}
\put(-15, 42){\line(0, 1){58}}
\put(-15, 0){\line(0, 1){36}}
}

\qbezier(-19, 68)(-50, 50)(-15, 38)

\qbezier(-11, 68)(-3, 70)(0, 85)

\qbezier(-15, 38)(-5, 38)(0, 20)

\put(20, 20){\line(0, 1){65}}

\qbezier(0, 20)(10, -10 ) (20, 20)
\qbezier(0, 85)(10, 105) (20, 85)
\end{picture}\right)= - tr_V\left( (1\otimes K)\check{R}_{V, M}\check{R}_{M, V}\right).
\]
It follows from \cite[Proposition 1]{ZGB} that $\Phi$ is a scalar multiple of $\id_M$. 
To compute the scalar, we consider the action of $\Phi$ on the highest weight vector $m_+$ of $M$. Using \eqref{eq:RtR}, we obtain
\[
\Phi m_+= - tr\left( \begin{pmatrix} q & 0\\ 0 & q^{-1}\end{pmatrix} \begin{pmatrix} q^{\ell} & q^{-1}(q-q^{-1})\\ 0 & q^{-\ell}\end{pmatrix}
\right) m_+=-(\Omega+\Omega^{-1})m_+.
\]
This leads to 
$
\Phi= -(\Omega+\Omega^{-1})\id_M,  
$
and hence 
 $\widehat\CF$ respects the tangled loop removal relation. 

This completes the proof of Theorem \ref{thm:functor-quot}. 
\end{proof}

\subsection{An equivalence of categories}

In this section, we take $V=V(1)$ and $M=M(\ell)$ for $\ell\ge -1$.  
Let $\CT$ be the full subcategory of the category $\CO_{int}$ of $\U_q(\fsl_2)$-modules with objects 
$M\otimes V^{\otimes r}$ for $r=0, 1, \dots$. Regard $\TLBC(q, q^{\ell+1})$, the Temperley-Lieb 
category of type $B$, as a full subcategory of $\ATLC(q, q^{\ell+1})$.

Theorem \ref{thm:RT} enables us to define the following functors. 
\begin{definition}
Let  $\CF: \ATLC(q,  q^{\ell+1})\longrightarrow \CO_{int}$ be the functor defined by the commutative diagram 
\[
\xymatrix{ 
\ATLC(q)  \ar@{->>}[d]\ar[r]^{\widehat\CF}& \CO_{int}\\
\ATLC(q, q^{\ell+1}).\ar[ur]_{\CF}&
}
\]

It is clear that $\CF$ sends objects and morphisms in $\TLBC(q, q^{\ell+1})$ to $\CT$.  Hence the 
restriction of the functor $\CF$ to $\TLBC(q, q^{\ell+1})$ leads  to a covariant functor 
\[
\CF':  \TLBC(q, q^{\ell+1})\longrightarrow \CT.
\]
\end{definition}

\begin{theorem}  \label{thm:main} Let $V=V(1)$ and $M=M(\ell)$. Then for all $\ell\ge -1$, the functor 
$
\CF':  \TLBC(q, q^{\ell+1})\longrightarrow \CT
$ 
is an equivalence of categories. 

\end{theorem}
\begin{proof} We will prove the equivalence of categories by showing  that the functor $\CF'$ is essentially surjective and fully faithful. 

We set $\Omega=q^{\ell+1}$ in this proof. 

The essential surjectivity is clear, since  $\CF'(m, v^r) = M\otimes V^{\otimes r}$ for all $r$.

Since $V$ is self dual, $\Hom_\CT(M\otimes V^{\otimes r}, M\otimes V^{\otimes s})\cong 
\Hom_\CT(M, M\otimes V^{\otimes (r+s)})$ as vector spaces for all $r$ and $s$. In view of Lemma \ref{lem:hom-iso}, we have 
$\Hom_{\TLBC(q, \Omega)}((m, v^r), (m, v^s))$ $\cong$  $\Hom_{\TLBC(q, \Omega)}(m, (m, v^{r+s}))$. Hence in order to prove that 
the functor $\CF'$ is fully faithful, it suffices to show that 
$\CF'$ defines isomorphisms $\Hom_{\TLBC(q, \Omega)}(m, (m, v^r))\stackrel{\sim}
\longrightarrow \Hom_\CT(M, M\otimes V^r)$ for all $r$.
We shall do this by showing that  for all $r$, 
\begin{eqnarray}
&\dim\Hom_\CT(M,  M\otimes V^{\otimes r}) = \dim\Hom_{\TLBC(q, \Omega)}(m, (m, v^r)), \label{eq:dim-1}\text{ and }\\
&\dim\CF'(\Hom_{\TLBC(q, \Omega)}(m, (m, v^r))) = \dim\Hom_{\TLBC(q, \Omega)}(m, (m, v^r)).  \label{eq:dim-2}
\end{eqnarray}

Consider first $\Hom_\CT(M,  M\otimes V^{\otimes r})$. We decompose $V^{\otimes r}$ 
with respect to the joint action of $\U_q(\fsl_2)$ and $\TL_r(q)$ to obtain $V^{\otimes r} =\bigoplus_t V(t)\otimes W_{t}(r)$, 
where the direct sum is over all $t$ such that  $0\le t\le r$ and $r-t$ is even.  Since $\ell\ge -1$, we can apply 
part (2) of Lemm \ref{lem:proj} to obtain
\[
\begin{aligned}
\Hom_{\U_q(\fsl_2)}(M,  M\otimes V^{\otimes r})
&=\bigoplus_t \Hom_{\U_q(\fsl_2)}(M,  M\otimes V(t))\otimes W_{t}(r)\\
&=\bigoplus_t V(t)_0\otimes W_t(r)\\
\end{aligned}
\]

If $r$ is odd, then all $t$ are odd, hence all $V(t)_0=0$.  
If $r=2N$ is even,  
\[
\begin{aligned}
\Hom_{\U_q(\fsl_2)}(M,  M\otimes V^{\otimes 2N})=\bigoplus_{t=0}^N  V(2t)_0\otimes W_{2t}(2N)
=\bigoplus_{t=0}^N  W_{2t}(2N). 
\end{aligned}
\]
In view of Lemma \ref{lem:dim-TLB} and its proof, this establishes equation \eqref{eq:dim-1}.  


Next we consider $\CF'(\Hom_{\TLBC(q, \Omega)}(m, (m, v^r)))$.

Applying $\CF'$ to the filtration \eqref{eq:filtr} of $W(2N)$ by $\TL_{2N}(q)$-modules $F_tW(2N)$,  and writing 
$\CF_iW(2N)=\CF'(F_iW(2N))$, we obtain 
\begin{eqnarray}\label{eq:filtr-1}
\CF_NW(2N)\supset \CF_{N-1}W(2N)\supset \dots\supset \CF_1W(2N)\supset \CF_0W(2N)\supset\emptyset.
\end{eqnarray}
This is a filtration of modules for $\CF'(\End^0(2N))$, which by \cite[Thm. 3.5]{LZBMW} is isomorphic to $\TL_r(2N)$.
For any $i$, if the quotient $W'_{2i}(2N):=\frac{\CF_{i}W(2N)}{\CF_{i-1}W(2N)}\ne 0$, then it must be 
isomorphic to the cell module $W_{2i}(2N)$ as $\TL_r(2N)$-module.  Therefore, if we can show that 
$W'_{2i}(2N)\ne 0$ for all $i$, then \eqref{eq:dim-2} follows in view of Lemma \ref{lem:dim-TLB}. 

Assume to the contrary that  $W'_{2i}(2N)=0$ for some $i$. This happens precisely if, given any 
distinguished diagram $D\in F_iW(2N)$ with $i$ thin arcs entangled with the pole, there is an element $D^{red}\in F_{i-1}W(2N)$ such that 
\begin{eqnarray}\label{eq:contra-1}
\CF'(D)-\CF'(D^{red})=0. 
\end{eqnarray}

Let $D$ be given by the distinguished diagram 
in Figure \ref{fig:0-2N}. 
We shall show that $\CF'(D)\not\in\CF'(F_{i-1}W(2N)$

\begin{figure}[h]
\begin{picture}(150, 105)(-20,0)
{
\linethickness{1mm}
\put(-15, 25){\line(0, 1){75}}
\put(-15, 0){\line(0, 1){16}}
}

\qbezier(10, 100)(6, 80)(-11, 75)
\qbezier(-2, 100)(-5, 80)(-11, 80)

\qbezier(-19, 74)(-30, 70)(-30, 50)
\qbezier(-15, 22)(-30, 26)(-30, 50)

\qbezier(-19, 80)(-40, 75)(-38, 40)
\qbezier(-15, 18)(-40, 25)(-38, 50)

\qbezier(-15, 18)(50, 5)(50, 100)
\qbezier(-15, 22)(30, 10)(30, 100)
\put(35, 80){...}
\qbezier(70, 100)(80, 45)(90, 100)
\put(95, 80){...}
\qbezier(110, 100)(120, 45)(130, 100)
\end{picture}
\caption{Diagram $D: m\to(m, v^{2N})$}
\label{fig:0-2N}
\end{figure}

\noindent
Take the morphisms 
$A: (m, v^{2N})\to (m, v^{2i})$,  $I_i: v^i\to v^i$ and $S: (m, v^{3i})\to (m, v^i)$, which are respectively given by 
\[
\begin{picture}(150, 65)(-20,0)
\put(-30, 30){$A=$}
{
\linethickness{1mm}
\put(0, 0){\line(0, 1){60}}
}
\put(10, 0){\line(0, 1){60}}
\put(15, 30){...}

\put(15, 15){$2i$}
\put(30, 0){\line(0, 1){60}}

\qbezier(40, 0)(50, 60)(60, 0)
\put(65, 15){...}
\qbezier(80, 0)(90, 60)(100, 0)
\put(105, 0){,}
\end{picture}
\quad
\begin{picture}(80, 65)(-20,0)
\put(-30, 30){$I_i=$}
\put(0, 0){\line(0, 1){60}}
\put(5, 30){...}

\put(8, 15){$i$}
\put(20, 0){\line(0, 1){60}}
\put(25, 0){,}
\end{picture}
\quad 
\begin{picture}(110, 65)(-20,0)
\put(-30, 30){$S=$}
{
\linethickness{1mm}
\put(0, 0){\line(0, 1){60}}
}
\put(10, 0){\line(0, 1){60}}
\put(15, 30){...}

\put(18, 15){$i$}

\put(30, 0){\line(0, 1){60}}

\qbezier(40, 0)(70, 80)(100, 0)
\qbezier(60, 0)(70, 60)(80, 0)
\put(45, 5){...}
\put(105, 0){.}
\end{picture}
\]
We first compose $D$ and the related $D^{red}$ with $A$ to obtain $A D$ and $AD^{red}$, then 
tensor them with $I_i$ to obtain $AD\otimes I_i$ and $AD^{red}\otimes I_i$ , and finally compose 
these morphisms with $S$ to obtain $S(AD\otimes I_i)$ and $S(AD^{red}\otimes I_i)$. 
Write $\tilde{D}= \delta^{-N+i}S(AD\otimes I_i)$ and $\tilde{D}^{red}= \delta^{-N+i}S(AD^{red}
\otimes I_i)$; these are both endomorphisms of $(m, v^i)$. 

Now if $W'_{2i}(2N)=0$ for some $i$, equation \eqref{eq:contra-1} implies
\begin{eqnarray}\label{eq:contrad}
\CF'(\tilde{D})-\CF'(\tilde{D}^{red})=0.
\end{eqnarray}

The diagram of $\tilde{D}$ is given by  Figure \ref{fig:i-i}. 
The morphism $\tilde{D}^{red}$ is spanned by diagrams with less than $i$ thin arcs tangled with the thick arc. 
\begin{figure}[h]
\begin{picture}(150, 105)(-20,0)
{
\linethickness{1mm}
\put(-15, 25){\line(0, 1){75}}
\put(-15, 0){\line(0, 1){16}}
}

\qbezier(20, 100)(6, 80)(-11, 75)
\qbezier(-2, 100)(-5, 80)(-11, 80)

\qbezier(-19, 74)(-30, 70)(-30, 50)
\qbezier(-15, 22)(-30, 26)(-30, 50)

\qbezier(-19, 80)(-40, 75)(-38, 40)
\qbezier(-15, 18)(-40, 25)(-38, 50)

\qbezier(-15, 18)(-5, 15)(0, 0)
\qbezier(-15, 22)(10, 15)(20, 0)
\put(0, 90){...}
\put(3, 95){$i$}

\end{picture}
\caption{Diagram $\tilde{D}: (m, v^i)\to(m, v^i)$}
\label{fig:i-i}
\end{figure}

Let ${\bf w}=m_+\otimes v_{-1}^{\otimes i}$, where $v_{-1}^{\otimes i}=\underbrace{v_{-1}
\otimes v_{-1}\otimes \dots\otimes v_{-1}}_i$ and let us compare $\CF'(\tilde{D})({\bf w})$ and $\CF'(\tilde{D}^{red})({\bf w})$. By \eqref{eq:RtR}, we have 
\[
\CF'(\tilde{D})({\bf w})=\sum_{k=0}^i\frac{(q-q^{-1})^k q^{-i \ell + k(\ell+i-2k)}}{\qint{k}!}F^k m_+\otimes E^k v_{-1}^{\otimes i}.
\]
In particular,  the vector $F^i m_+\otimes v_1^{\otimes i}$ appears in $\CF'(\tilde{D})({\bf w})$ with a 
nonzero coefficient.  This vector is nonzero since $F^i m_+\ne 0$ for all $i$ in the Verma module $M$. 

Turning to $\CF'(\tilde{D}^{red})({\bf w})$,  we note that ${\bf w}=m_+\otimes v_{-1}^{\otimes i}$ 
is annihilated by the images under $\CF'$ of all diagrams from $(m, v^i)$ to $(m, v^{i-2})$ which have an arc $U$
as depicted below, since the invariant form on $V(1)$ satisfies $(v_{-1},v_{-1})=0$.
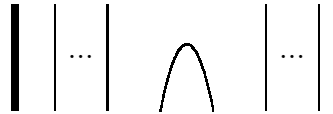
\[
\begin{picture}(150, 50)(-35,0)
{
\linethickness{1mm}
\put(-15, 0){\line(0, 1){40}}
}
\put(0, 0){\line(0, 1){40}}
\put(20, 0){\line(0, 1){40}}
\put(5, 20){...}
\qbezier(40, 0)(50, 50)(60, 0)
\put(80, 0){\line(0, 1){40}}
\put(100, 0){\line(0, 1){40}}
\put(85, 20){...}
\put(105, 0){. }
\end{picture}
\]
Thus among all the diagrams in $\tilde{D}^{red}$, only those shown in Figure \ref{fig:less-i} with $t<i$  could
have nonzero contributions to $\CF'(\tilde{D}^{red})({\bf w})$. 
\begin{figure}[h]
\begin{picture}(150, 105)(-60,0)
\put(-65, 50){$\Upsilon_t =$}
{
\linethickness{1mm}
\put(-15, 25){\line(0, 1){75}}
\put(-15, 0){\line(0, 1){16}}
}

\qbezier(20, 100)(6, 80)(-11, 75)
\qbezier(-2, 100)(-5, 80)(-11, 80)

\qbezier(-19, 74)(-30, 70)(-30, 50)
\qbezier(-15, 22)(-30, 26)(-30, 50)

\qbezier(-19, 80)(-40, 75)(-38, 40)
\qbezier(-15, 18)(-40, 25)(-38, 50)

\qbezier(-15, 18)(-5, 15)(0, 0)
\qbezier(-15, 22)(10, 15)(20, 0)
\put(0, 90){...}
\put(5, 95){$t$}

\put(30, 0){\line(0, 1){100}}
\put(40, 90){......}
\put(70, 0){\line(0, 1){100}}
\put(40, 95){$i-t$}
\end{picture}
\caption{A diagram in $\tilde{D}^{red}$}
\label{fig:less-i}
\end{figure}
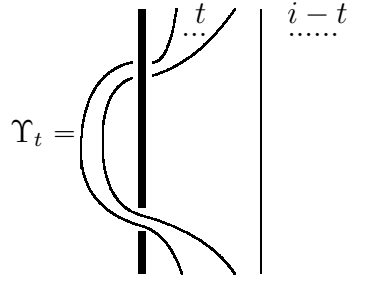

\noindent
Using \eqref{eq:RtR}, we obtain  
\[
\CF'(\Upsilon_t )({\bf w}) =\sum_{k=0}^t \frac{(q-q^{-1})^k q^{-t \ell + k(\ell+t-2k) }}{\qint{k}!}F^k m_+
\otimes E^k v_{-1}^{\otimes t} \otimes v_{-1}^{\otimes (i-t)}.
\]
Note that for all $t<i$, the vector $F^i m_+\otimes v_1^{\otimes i}$ never appears in $\CF'(\Upsilon_t )({\bf w})$ with a non-zero coefficient.
Thus \eqref{eq:contrad} does not hold for any element ${D}^{red}\in F_{i-1}W(2N)$. Hence $W'_{2i}(2N)=W_{2i}(2N)$  for all $i$, and equation \eqref{eq:dim-2} is proved. 

We have now shown that $\CF'$ is fully faithful, which completes the proof of  Theorem \ref{thm:main}. 
\end{proof}

As an immediate consequence, we have the following result.
\begin{corollary} \label{cor:alg-iso} Let $V=V(1)$ and $M=M(\ell)$ with $\ell\ge -1$. Then for each $r=1, 2, \dots$, 
there is an isomorphism of associative algebras
\[
\TLB_r(q, \sqrt{-1} q^{\ell+1})\stackrel{\sim}\longrightarrow \End_{\U_q(\fsl_2)}(M\otimes V^{\otimes r}). 
\]
\end{corollary}
\begin{proof} 
The functor $\CF'$ of Theorem \ref{thm:main} leads to  an isomorphism of associative algebras 
$\TLBC_r(q, q^{\ell+1})\stackrel{\sim}\longrightarrow \End_{\U_q(\fsl_2)}(M\otimes V^{\otimes r})$ for each $r$, where $\TLBC_r(q, q^{\ell+1})$ is defined by Definition \ref{def:algs}.  
By Lemma \ref{lem:TLB-alg-iso}, $\TLBC_r(q, q^{\ell+1})\cong \TLB_r(q, \sqrt{-1}q^{\ell+1})$. This proves the corollary.
\end{proof}

\section{An alternative version of the category $\TLBC$.}\label{sect:TLB-old}

\subsection{The category $\TLBB(q,Q)$}
Let $R$ be a ring and let $q,Q\in R$ be invertible elements. We begin by recalling 
the definition of the category $\TLBB(q,Q)$ from \cite{GL03}. The objects of $\TLBB(q,Q)$ are
the integers $t\in\Z_{\geq 0}$. A ($\TLBB$) diagram $D:t\lr s$ is a ``marked Temperley-Lieb diagram'' from $t$ to $s$,
and $\Hom_{\TLBB(q,Q)}(t,s)$ is the free $R$ module with basis the set of $\TLBB$ diagrams: $t\lr s$.

\begin{figure}[h]
\begin{center}
\begin{tikzpicture}[scale=1.5]

\foreach \x in {2,3,4,5}
\filldraw(\x,0) circle (0.05cm);
\foreach \x in {1,2,3,4,5,6}
\filldraw(\x,2) circle (0.05cm);
\draw (2,0)--(5,2);
\draw (5,0)--(6,2);

\draw node at (1,1) {Left region};

\draw [dashed] (0,0)--(7,0); \draw [dashed] (0,2)--(7,2);
\draw [dashed] (0,0)--(0,2);\draw [dashed] (7,0)--(7,2);



\draw(1,2) .. controls (2,1.2) and (3,1.2) .. (4,2);
\draw(2,2) .. controls (2.2,1.7) and (2.8,1.7) .. (3,2);

\draw(3,0) .. controls (3.2,0.3) and (3.8,0.3) .. (4,0);


\end{tikzpicture}
\end{center}
\caption{}\label{fig-2}
\end{figure}

To define marked diagrams, recall that a $\TL$ diagram $t\lr s$ divides the ``fundamental rectangle''  into regions,
of which there is a unique leftmost one (see Fig. \ref{fig-2}, which is a $\TL$ diagram: $4\lr 6$). A marked diagram is
a $\TL$ diagram in which the boundary arcs of the left region may be marked with dots. Thus the $\TL$ 
diagram depicted in Fig.1 has $2$ markable arcs, and a corresponding marked diagram is shown in Fig. \ref{fig-3}.

\begin{figure}[h]
\begin{center}
\begin{tikzpicture}[scale=1.5]

\foreach \x in {2,3,4,5}
\filldraw(\x,0) circle (0.05cm);
\foreach \x in {1,2,3,4,5,6}
\filldraw(\x,2) circle (0.05cm);
\draw (2,0)--(5,2);
\draw (5,0)--(6,2);

\draw node at (1,1) {Left region};

\draw [dashed] (0,0)--(7,0); \draw [dashed] (0,2)--(7,2);
\draw [dashed] (0,0)--(0,2);\draw [dashed] (7,0)--(7,2);



\draw(1,2) .. controls (2,1.2) and (3,1.2) .. (4,2);
\draw(2,2) .. controls (2.2,1.7) and (2.8,1.7) .. (3,2);

\filldraw(2.5,1.4) circle (0.1cm);\filldraw(3,0.66) circle (0.1cm);\filldraw(4,1.33) circle (0.1cm);

\draw(3,0) .. controls (3.2,0.3) and (3.8,0.3) .. (4,0);


\end{tikzpicture}
\end{center}
\caption{}\label{fig-3}
\end{figure}

These marked diagrams are composed via concatenation, just as $\TL$-diagrams, with three rules to bring them 
to ``standard form'', which is a diagram with at most one mark on each eligible arc. The rules are (recalling that
for any invertible element $x\in R$, $\delta_x=-(x+x\inv)$:

\be\label{eq:tlbrules}
\begin{aligned}
&\text{(i) If $D$ is a (marked) diagram and $L$ is a loop with no marks then $D\amalg L=\delta_q D$.}\\
&\text{(ii) If, in (i), $L$ is a loop with 1 mark then $D\amalg L=\left(\frac{q}{Q}+\frac{Q}{q}\right) D$.}\\
&\text{(iii) If an arc of a diagram $D$ has more than $1$ mark and $D'$ is obtained from $D$}\\
&\text {by removing one mark from that arc, then $D=\delta_Q D'$.}\\
\end{aligned}
\ee

\begin{definition}\label{def:tlbb}
The category $\TLBB(q,Q)$ has objects $\Z_{\geq 0}$ and morphisms which are $R$-linear combinations of marked diagrams,
subject to the rules \eqref{eq:tlbrules}.
\end{definition}

It is evident that $\Hom_{\TLBB}(s,t)$ has basis consisting of $\TLB$ (marked) diagrams with at most one mark
on each boundary arc of the left region. We shall refer to these as ``marked diagrams''; the next result counts them. 

\begin{proposition}\label{prop:tlbdim}
For integers $t,k\geq 0$, the number $b(t,t+2k)$ of marked diagrams $t\lr t+2k$ depends only on $t+k$. If $t+k=m$, then 
the number of these is $d(m)=\binom{2m}{m}$.
\end{proposition}
\begin{proof}
Given a marked diagram $t\lr t+2k$, one obtains a diagram $0\lr 2(t+k)$ by rotating the bottom of the diagram through 
$180^\circ$ until it becomes part of the top, pulling all the relevant arcs appropriately. This is illustrated below in Fig. \ref{fig-4}  for the diagram
in Fig. \ref{fig-3}. (but note that we are applying this construction only to standard diagrams, which have at most one mark on each arc).

\begin{figure}[h]
\begin{center}
\begin{tikzpicture}[scale=1.5]

\foreach \x in {1,2,3,4,5,6,7,8,9,10}
\filldraw(\x,2) circle (0.05cm);





\draw(1,2) .. controls (2,1.2) and (3,1.2) .. (4,2);
\draw(2,2) .. controls (2.2,1.7) and (2.8,1.7) .. (3,2);
\draw(6,2) .. controls (6.2,1.7) and (6.8,1.7) .. (7,2);
\draw(8,2) .. controls (8.2,1.7) and (8.8,1.7) .. (9,2);
\draw(5,2) .. controls (6.2,1) and (8.8,1) .. (10,2);

\filldraw(2.5,1.4) circle (0.1cm);\filldraw(6.5,1.33) circle (0.1cm);\filldraw(8.5,1.33) circle (0.1cm);



\end{tikzpicture}
\end{center}
\caption{}\label{fig-4}
\end{figure}

This shows that $b(t,k)=b(0,t+k)$, which proves the first statement. Write $d(m)=b(0,2m)$, and following \cite[\S 4.3]{ILZ1},
we write $d(x)=\sum_{i=0}^\infty d(i)x^i$, where $d(0)=1$. Now if the $2m$ upper dots are numbered $1,2,\dots, 2m$
from left to right any marked diagram $D:0\lr 2m$ will join $1$ to an even numbered dot, say $2i$. For fixed $i$, the number
of such $D$ is $2c(i-1)d(m-i)$, where $c(i)$ is the Catalan number in \cite{ILZ1} (since the arc $(1,2i)$ may be either
marked or unmarked). It follows that 
\be\label{eq:rec1}
d(m)=2\sum_{i=1}^m c(i-1)d(m-i).
\ee
Multiplying \eqref{eq:rec1} by $x^m$ and summing, we obtain
\be\label{eq:rec2}
d(x)=\frac{1}{1-2xc(x)},
\ee
where $c(x)=\sum_{i=0}^\infty c(i)x^i=\sum_{i=0}^\infty \frac{1}{i+1}\binom{2i}{i}x^i$.
Now from the above relation, one sees easily that
\be\label{eq:rec3}
\frac{\partial}{\partial x}(xc(x))=xc'(x)+c(x)=\sum_{n=0}^\infty \binom{2n}{n}x^n.
\ee
Now differentiating the relation $xc(x)^2=c(x)-1$, we obtain 
\[
(1-2xc(x))\inv=d(x)=\frac{c'(x)}{c(x)^2}=\frac{\partial}{\partial x}(xc(x)-1),
\]
and the result follows from \eqref{eq:rec3}.
\end{proof}

\subsection{Generators and cellular structure} \label{ss:tlbb}
The usual Temperley-Lieb category $\TLC(q)$ (see Definition \ref{def:tl(b)}) is a subcategory of $\TLBB(q,Q)$, with the same objects, and morphisms
which are $R$-linear combinations of unmarked diagrams. Thus there is a faithful functor 
\[
\TLC(q)\lr\TLBB(q,Q)
\] 
as well as a ``tensor product'' functor
\be\label{eq:tp1}
\TLBB(q,Q)\times \TLC(q)\lr\TLBB(q,Q),
\ee
given by juxtaposing diagrams. Note that the functor \eqref{eq:tp1} restricts to the usual tensor product on $\TLC(q)$.

The category $\TLC(q)$ is generated, under composition and tensor product by the morphisms $A,U,I$ depicted in 
Fig. \ref{fig-5}  below, subject to the obvious relations. 

\begin{figure}[h]
\begin{center}
\begin{tikzpicture}[scale=1.5]

\foreach \x in {1,2,7,9}
\filldraw(\x,0) circle (0.05cm);

\foreach \x in {4,5,7,9}
\filldraw(\x,2) circle (0.05cm);
\draw (7,0)--(7,2);
\draw (9,0)--(9,2);

\draw node at (3,1) {;};\draw node at (6,1) {;};\draw node at (8,1) {;};

\filldraw(9,1) circle (0.1cm);



\draw node at (1.5,-0.5) {A};\draw node at (4.5,-0.5) {U};\draw node at (7,-0.5) {I};\draw node at (9,-0.5) {$C_0$};

\draw(4,2) .. controls (4.2,0) and (4.8,0) .. (5,2);

\draw(1,0) .. controls (1.2,1.8) and (1.8,1.8) .. (2,0);

\end{tikzpicture}
\end{center}
\caption{}\label{fig-5}
\end{figure}

The category $\TLBB(q,Q)$ is generated by the generators $A,U,I$ of $\TLC$, with $C_0$ added as shown. Evidently 
it follows from \eqref{eq:tlbrules}(iii) that $C_0$ satisfies $C_0^2=\delta_Q C_0$ 
and from \eqref{eq:tlbrules}(ii) that $A(C_0\ot I)U=\frac{q}{Q}+\frac{Q}{q}$. For $n=1,2,\dots$, the algebra $\Hom_{\TLBB(q,Q)}(n,n)$
is the Temperley-Lieb algebra $\TLBB_n(q,Q)$ of type $B_n$. It has a cellular structure described as follows. 

Given $n\in\Z_{>0}$, define $\Lambda_B(n)=\{t\in\Z\mid |t|\leq n\text{ and }t\equiv n\text{(mod }2)\}$. For $t\in\Lambda_B(n)$
define $M(t)$ as the free $R$-module with basis the monic diagrams $D:t\lr n$ in which no through string is marked. 
For each $t\in\Lambda_B (n)$, there is an injective map $\beta_t:M(t)\times M(t)\lr \TLBB_n(q,Q)$ given by 
\be\label{eq:cell}
\beta_t(D_1,D_2)=
\begin{cases}
D_2^*D_1\text{ if }t\geq 0\\
D_2^*(C_0\ot I^{\ot(|t|-1)})D_1\text{ if }t<0,\\
\end{cases}
\ee
where $D^*$ is the diagram obtained from $D$ by reflection in a horizontal.

The next result is straightforward.

\begin{proposition}\label{prop:tlbcell}
Maintain the above notation.
\begin{enumerate}
\item Let $C:=\amalg_{t\in\Lambda_B(n)}\beta_t:\amalg_{t\in\Lambda_B(n)}M(t)\times M(t)\lr \TLBB_n(q,Q)$. The image 
of $C$ is a basis of $\TLBB_n(q,Q)$. Write $\beta_t(S,T)=C_{S,T}^t$ for the basis elements ($S,T\in M(t)$).
\item The basis $C^t_{S,T}$ ($t\in\Lambda_B(n)$, $S,T\in M(t)$) is a cellular basis of $\TLBB_n(q,Q)$.
\end{enumerate}
\end{proposition}

\begin{remark}
An analogous result for Temperley-Lieb algebras of type $D$ is proved in \cite{LS}
\end{remark}

\subsection{Cell modules for $\TLBB_n(q,Q)$} We give a description of the cell modules corresponding to the cellular structure
given in Proposition \ref{prop:tlbcell} and compute their dimension, although this is implicit in the results of \cite{GL03}.
The cell module $W_t(n)$  ($t\in\Lambda(n)$) has basis the set $M(t)$, with $\TLBB_n(q,Q)$-action defined in the usual way
by multiplication of diagrams.

Let $u(t,k)=\rank(W_t(|t|+2k))$. We have seen (Proposition \ref{prop:tlbdim}) that $u(0,k)=\binom{2k}{k}$. This is the case $t=0$
of the following result.

\begin{proposition}\label{prop:dimwt}
We have 
\[
\dim W_t(|t|+2k)=u(t,k)=\binom{|t|+2k}{k}.
\]
\end{proposition}
\begin{proof}
It clearly suffices to consider the case $t\geq 0$. We prove the result by induction on the pair $t,k$, 
the result being known for $t=0$ (Proposition \ref{prop:tlbdim}), while for $k=0$, clearly $u(t,0)=1=\binom{|t|}{0}$. 

Now the same argument
as in \cite[Prop. 5.2]{ILZ1}, involving rotation of the bottom row of a monic diagram $|t|\lr |t|+2k$ through $180^\circ$
to obtain a diagram $0\lr 2|t|+2k$, shows that we have the following recursion for $u(t,k)$. Assume $t,k\geq 1$. Then
\be\label{eq:urec}
u(t,k)=u(t-1,k)+u(t+1, k-1).
\ee
Hence by induction, $u(t,k)=\binom{t-1+2k}{k}+\binom{t+1+2(k-1)}{k-1}=\binom{t+2k}{k}$.
\end{proof}

\subsection{An equivalence of categories} We have defined two ``Temperley-Lieb categories of type $B$'', viz. the category 
$\TLBC(q,\Omega)$ of Definition \ref{def:tl(b)} and the category $\TLBB(q,Q)$ of Definition \ref{def:tlbb}. 
Both categories contain the finite \tl category $\TLC(q)$ as a subcategory. In the case of $\TLBB(q,\Omega)$ this is realised as in 
\S\ref{ss:tlbb}. In the case of $\TLBC(q,\Omega)$ (cf. Definition \ref{def:tl(b)}) $\TLC(q)$ may be thought of as having the 
same objects $\{(m,v^r)\}, r=0,1,\dots$ as $\TLBC(q,\Omega)$, but where the morphisms are linear combinations of tangles which are
not entwined with the pole.

Our next objective is to prove the following result.

\begin{theorem}\label{thm:tlbequ}
Let $R$ be an integral domain with invertible elements $q,Q$ and $\Omega$ and an element $\sqrt{-1}$ such that
$\is^2=-1$. Then there is an equivalence of categories $\CM:\TLBB(q,Q)\lr \TLBC(q,\Omega)$ which takes 
the object $r\in\Z$ to $(m,v^r)$, is the identity on $\TLC(q)$, and respects the tensor product, if and only if 
$\Omega=\pm (\is Q)^{\pm 1}$. In this case we have $\cM(C_0)=\sqrt{-1}qL-QI$.
\end{theorem}

Note that the stated conditions on $\CM$ imply that for diagrams
$D\in\TLBB(q,Q)$ and $D'\in\TLC(q)$, we have
\be\label{eq:resp}
\CM(D\ot D')=\CM(D)\ot D'.
\ee

\begin{proof}[Proof of Theorem \ref{thm:tlbequ}]
We shall define $\CM$ on the generators $A,U,I$ and $C_0$, the effect of $\CM$ on objects having been given.
Since $\CM$ is to be the identity functor on the subcategory $\TLC(q)$, evidently we must have $\CM(A)=A$,
$\CM(U)=U$ and $\CM(I)=I$, where on the left side of these equations $A,U$ and $I$ are as in Fig. \ref{fig-5} of this section,
while on the right side they are as defined in \S\ref{sss:tp}.

It remains only to define $\CM(C_0)$. This is a morphism in $\End_{\TLBC(q,\Omega)}(m,v)$, and since this space 
has basis $I,L$, it follows that
\be\label{eq:b1}
\CM(C_0)=aL+bI,
\ee
for $a,b\in R$. We shall determine constraints on $a,b$. First, observe that it follows
by applying $A(-\ot I)U$ to both sides of \eqref{eq:b1} that
\[
\CM(A(C_0\ot I)U)=aA(L\ot I)U+bA(I\ot I)U,
\]
whence using \eqref{eq:loop} and \eqref{eq:tlbrules}(ii) it follows that
\be\label{eq:b2}
\kappa(= \frac{q}{Q}+\frac{Q}{q})=a\delta_\Omega+b\delta_q,
\ee
where, for any invertible $x\in R$, $\delta_x=-(x+x\inv)$.

Next, we square both sides of \eqref{eq:b1} using the relations \eqref{eq:skein} and \eqref{eq:tlbrules}(iii).
One obtains
\be\label{eq:b3}
(2ab-q\inv\delta_\Omega a^2)L+(b^2-q^{-2}a^2)=a\delta_Q L+b\delta_Q I,
\ee
and equating the coefficients of $L$ and $I$ respectively, we obtain
\be\label{eq:b4}
2ab-q\inv\delta_\Omega a^2=a\delta_Q,
\ee
and
\be\label{eq:b5}
b^2-q^{-2}a^2=b\delta_Q.
\ee
Moreover, since evidently $a\neq 0$, we may divide \eqref{eq:b4} by $a$ to obtain
\be\label{eq:b6}
2b-q^{-1}\delta_\Omega a=\delta_Q.
\ee

It therefore remains to solve equations \eqref{eq:b2}, \eqref{eq:b5} and \eqref{eq:b6} for $a$ and $b$.

It is straightforward to show that \eqref{eq:b6} and \eqref{eq:b2} imply that
\be\label{eq:b7}
b=-Q.
\ee

It now follows easily from \eqref{eq:b2} that
\be\label{eq:b8}
a=\frac{q(Q\inv-Q)}{\delta_\Omega}.
\ee

Now the values of $a,b$ in \eqref{eq:b7} and \eqref{eq:b8} are easily shown to satisfy \eqref{eq:b2} and \eqref{eq:b6}.
However they satisfy \eqref{eq:b5} if and only if $\delta_\Omega^2=-(Q\inv-Q)^2$, and this holds if and only if
\begin{eqnarray}\label{eq:Omega-Q}
\Omega=\pm \is Q^{\pm 1}. 
\end{eqnarray}
Substituting this into \eqref{eq:b8}, we obtain 
\[
a=\pm\sqrt{-1} q.
\]
It is now easily checked that the defining relations among the generators $A,U,I$ and $C_0$, which are those involving only
$A,U$ and $I$, as well as those in \eqref{eq:tlbrules}, are respected by $\CM$ and the theorem is proved.
\end{proof}

The algebra $\TLBB_r(q,Q)$, defined above as the algebra of endomorphisms of $r$ in the category $\TLBB(q,Q)$
is generated \cite[(5.7)]{GL03} by the elements $c_i:=I^{\ot(i-1)}\ot(U\circ A)\ot I^{\ot(r-i-1)}$ ($i=1,2,\dots,r-1$) and
$c_0:=C_0\ot I^{\ot(r-1)}$, subject to the relations set out in \cite[Prop. (5.3)]{GL03}. Likewise, the endomorphism algebra
$\TLBC_r(q,\Omega)$ of the object $(m,v^r)$ in $\TLBC(q,\Omega)$ has generators $C_i$, $i=1,\dots,r-1$, defined in analogy with
the $c_i$ using the elements $I,A$ and $U$ of \S\ref{sss:tp}, as well as the element $L\ot I^{\ot(r-1)}$, which we refer to as $L\in\TLBC_r(q,\Omega)$.

\begin{corollary}\label{cor:tlb}
For any integer $r>0$, there is an isomorphism of algebras 
$$
\TLBB_r(q,Q)\lr \TLBC_r(q,\sqrt{-1} Q\inv),
$$
which, in the notation explained above, 
takes the generators $c_i$ to $C_i$ ($i=1,\dots,r-1$) and takes $c_0$ to $-\sqrt{-1}qL-Q I^{\ot r}$,
where $\sqrt{-1}$ is a fixed square root of $-1$.
\end{corollary}

\begin{remark}\label{rem:tlb-algs}
It follows from Corollary \ref{cor:tlb} and Lemma \ref{lem:TLB-alg-iso} that there is an isomorphism of algebras 
\[\TLBB_r(q,Q)\overset{\sim}{\lr}\TLB_r(q,Q).
\]
We mention also that evidently $\TLB_r(q,Q)\cong \TLB_r(q,-Q)$ via the isomorphism defined by $c_i\mapsto c_i$
($i=1,\dots,r-1$), $x_1\mapsto -x_1$. This accounts for the ambiguity in $\sqrt{-1}$.
\end{remark}

\subsection{Semisimplicity}
It is apparent that analysis of the algebras $\TLBB_n(q,Q)$ may be approached through their cellular structure 
outlined above.
This makes it possible to analyse the representations $M(\ell)\ot V(1)^{\ot r}$. In this section we 
determine precisely when $\End_{\U_q(\fsl_2)}(M(\ell)\ot V(1)^{\ot r})$ is semisimple. For ease of exposition, we assume 
throughout this section that $Q$ and $\Omega$ are both in the field $\CK_0$ and we have chosen a fixed square root $\sqrt{-1}\in\CK_0$.

\subsubsection{Semisimplicity of $\TLBB(q,\Omega)$} This may be approached as in \cite{CGM} for positive characteristic. 
However here we shall use the approach of \cite{GL98,GL03} which relate $\TLBB_r(q, Q)$ to the (unextended) affine Temperley-Lieb algebra
$T^a_r(q)$ as in \cite{GL03}, as well as the complete analysis of its cell modules $W_{t,z}(r)$ given in \cite{GL98}. 
The results we quote as background may all be found in \cite[\S\S6,10]{GL03}.

For each integer $r\geq 0$ define the following sets of parameters:
\be\label{eq:params}
\begin{aligned}
\Lambda(r)&=\{t\in\Z\mid 0\leq t\leq r \text{ and }r-t\in 2\Z\}\\
\Lambda_B(r)&=\{t\in\Z\mid |t|\in\Lambda(r)\}\\
\Lambda^a(r)&=\{(t,z)\in\Z\times \CK_0^*\mid t\in\Lambda(r)\}/\sim,\\
\end{aligned}
\ee
where in the third line, we declare $(0,z)\sim(0,z\inv)$.

The sets $\Lambda(r)$ and $\Lambda_B(r)$ are posets, with $\Lambda(r)$ being ordered in the obvious way, $\Lambda_B(r)$ ordered according
to $|t|$, with $|t|\geq t$. The set $\Lambda^a(r)$ indexes the cell modules of $T^a_r(q)$, among which all homomorphisms are known.

Any cellular algebra is semisimple if and only if there are no non-trivial homomorphisms among its cell modules \cite{GL96}. The key to the semisimplicity
of $\TLBB_r(q,Q)$ is therefore the following result. For the notation, the reader is referred to \cite[\S 5]{GL03}.

\begin{theorem}\label{thm:pullback}\cite[Cor (5.11), Thm. (6.15)]{GL03}
 \begin{enumerate}
\item For $1\leq i\leq r-1$, let $t_i=c_i+q\in\TLBB_r(q,Q)$, and let $t_0=c_0+Q$. There is a surjective homomorphism $g:T^a_r(q)\lr\TLBB_r(q,Q)$
defined by $g(f_i)=c_i$ for $i=1,\dots,r-1$, and $g(\tau)=\sqrt{-1}q^{\half(n-2)}t_0t_1\dots t_{r-1}$, where $\tau$ is the ``twist'' diagram in $T^a_r(q)$.
\item Denoting by $g^*(M)$ the pullback to $T^a_r(q)$ of a $\TLBB_r(q,Q)$-module $M$, we have, for $t\in\Lambda_B(r)$,
\[
g^*(W_t(r))\cong W_{|t|,z_t^{\varepsilon_t}}(r),
\]
where $\varepsilon_t=\frac{t}{|t|}=\pm 1$ and $z_t=(-1)^{t+\half}q^{-\frac{t}{2}}Q\inv$.
\end{enumerate}
\end{theorem}

Since all homomorphisms among the modules $W_{t,z}$ are known, Theorem \ref{thm:pullback} may be used to determine
whether $\TLBB_r(q,Q)$ is semisimple, because $\TLBB$-homomorphisms among the $W_t(r)$ are precisely
$T^a$-homomorphisms among the lifts. We begin by explaining when we have a non-trivial homomorphism between two
cell modules $W_{t,z}$.

Define a preorder on $\Lambda^a(r)$ as follows. Say that $(t,z)\prec (s,y)$ ($(t,z),(s,y)\in\Lambda^a(r)$) if for some $\ve=\pm 1$
we have
\be\label{eq:preorder}
\begin{aligned}
s=t+2m &\text {  for some }m>0\text{ and }\\
y&=q^{-\ve m}z\text{ and }\\
z^2&=q^{\ve s}.\\
\end{aligned}
\ee

A short calculation using the equations \eqref{eq:preorder} reveals that there is a non-zero homomorphism of cell modules
$W_s(r)\lr W_t(r)$ for $\TLBB_r(q,Q)$ ($t,s\in\Lambda_B(r)$) if and only if either $W_t(r)\cong W_{-t}(r)$ (see Corollary
\ref{cor:cellpm} for some $t>0$ or if the following conditions hold:

\be\label{eq:homs}
\begin{aligned}
(i) \;&\text{$\exists t,s\in\Lambda(r)$ such that }s=t+2m>t\geq 0, \text{ and }Q=\sqrt{-1}q^{-(t+m)};\\
(ii) \;&\text{$\exists t< 0$, }s>0\in\Lambda_B(r), \text{ such that }t=-2m,\;\;s=4m\text{ and }Q=\sqrt{-1}q^{m};\\
(iii) \;\;&\text{$\exists t< 0$, }s<0\in\Lambda_B(r), \text{ such that }|s|= |t|+2m>|t|\text{ and }Q=\sqrt{-1}q^{-m}.\\
\end{aligned}
\ee

\begin{corollary}\label{cor:cellpm}
With the above notation, for each $t>0$, there is a non-trivial homomorphism $:W_t(r)\lr W_{-t}(r)$
  if and only if  $W_t(r)\cong W_{-t}(r)$.
Moreover this condition is satisfied if and only if $Q=\sqrt{-1}$.
\end{corollary}
\begin{proof}
It follows from Theorem \ref{thm:pullback} that there is a homomorphism $:W_t(r)\lr W_{-t}(r)$ if and only
if there is a non-trivial homomorphism of $T^a_r(q)$-modules $:W_{t,z_t}(r)\lr W_{t,z_{-t}\inv}(r)$. But this happens 
if and only if  $W_{t,z_t}(r)\cong W_{t,z_{-t}\inv}(r)$, whence the first statement.

Moreover, again by the above statement, $W_t(r)\cong W_{-t}(r)$ if and only if $z_t=z_{-t}\inv$.
Using the value of $z_t$ given in Theorem \ref{thm:pullback} (2), one sees easily that this happens 
if and only if $Q^2=-1$.
\end{proof}

\subsection{Semisimplicity and Schur-Weyl duality} It is evidently a consequence of Theorem \ref{thm:main}, Corollary \ref{cor:alg-iso}
and  Corollary \ref{cor:tlb} that:

\begin{proposition}\label{prop:end-tlb}
Let $\ell\geq -1$ be an integer. For each integer $r\geq 1$, we have an isomorphism of associative algebras
\[
\End_{\U_q(\fsl_2)}(M(\ell)\ot V(1)^{\ot r})\overset{\sim}{\lr}\TLBB_r(q,\sqrt{-1}q^{-(\ell+1)}).
\]
\end{proposition}

Our final result uses the cellular structure to give a precise criterion for the semisimplicity of the endomorphism algebra, 
which may also be deduced from results in \cite{BS3}.

\begin{theorem}\label{thm:ss} Assume that $q$ is not a root of unity in $\CK_0$.
The endomorphism algebra $\End_{\U_q(\fsl_2)}(M(\ell)\ot V(1)^{\ot r})$ is non-semisimple for all $r$ if $\ell=-1$.
For $\ell\geq 0$, it is semisimple if and only if $r\leq\ell+1$. 
\end{theorem}
\begin{proof}
When $\ell=-1$, $Q=\sqrt{-1}$. Hence by Corollary \ref{cor:cellpm}, there are coincidences among the
cell modules of $\TLBB_r(q,Q)$, whence the endomorphism algebra is neither semisimple nor quasi-hereditary. 
This proves the first statement.

Now assume $\ell>-1$. We apply the criteria in \eqref{eq:homs} in the case where $Q=\sqrt{-1}q^{-(\ell+1)}$.
For the criterion (i) to apply, we require $t+m=\ell+1>0$; (ii) cannot apply in any case, while for 
(iii) we require $m=\ell+1>0$. 

Now in case (i) we require $t+2m=\ell+1+m\leq r$ for some $m>0$, so $r\geq \ell+2$.
In case (iii), we require $m=\ell+1$ and $|t|+2m\leq r$, i.e. $r> 2(\ell+1)\geq \ell+3$.  
This shows that $\End_{\U_q(\fsl_2)}(M(\ell)\ot V(1)^{\ot r})$ is semisimple for $r\geq\ell+1$.

Consider now the case $r=\ell+2$ (where $\ell\geq 0$). We show that there is always a non-trivial homomorphism 
$:W_{\ell+2}(\ell+2)\lr W_{(\ell)}(\ell+2)$. For this, observe that in the notation above,
 $z_\ell=(-1)^{\ell+\frac{3}{2}}q^{\frac{\ell}{2}+1}\sqrt{-1}$ and 
 $z_{\ell+2}=(-1)^{\ell+\frac{3}{2}}q^{\frac{\ell}{2}+1}\sqrt{-1}=q\inv z_\ell$.
 Now take $t=\ell,m=1$ and $\ve=1$ in \eqref{eq:preorder}. one concludes that $(\ell,z_\ell)\prec(\ell+2,z_{\ell+2})$.
 This proves that $\End_{\U_q(\fsl_2)}(M(\ell)\ot V(1)^{\ot r})$ is not semisimple for $r\geq \ell +2$ and $r\equiv\ell\text{(mod }2)$.
 
 A similar argument applies for $r\geq\ell+2$ with $r\equiv\ell+1\text{(mod }2)$, and the proof is complete.
\end{proof}

\end{document}